\documentclass[11pt]{article}
\usepackage{amsmath,amsfonts,amssymb,amsthm}
\allowdisplaybreaks[1]
\usepackage{mathtools}
\numberwithin{equation}{section}

\newtheorem{theorem}{Theorem}[section]
\newtheorem{definition}[theorem]{Definition}
\newtheorem{proposition}[theorem]{Proposition}
\newtheorem{lemma}[theorem]{Lemma}
\newtheorem{corollary}[theorem]{Corollary}
\newtheorem{remark}[theorem]{Remark}
\newtheorem{example}[theorem]{Example}

\newcommand{\V}{\operatorname{var}}

\newcommand{\R}{\mathbb{R}}
\newcommand{\NN}{\mathbb{N}}

\newcommand{\cC}{\mathcal{C}}
\newcommand{\cB}{\mathcal{B}}

\newcommand{\BP}{\mathbb{P}}

\newcommand{\BE}{\mathbb{E}}

\newcommand{\BQ}{\mathbb{Q}}
\newcommand{\cN}{\mathcal{N}}
\newcommand{\bN}{{\mathbf{N}}}

\newcommand{\dd}{\mathrm{d}} 
\newcommand{\BH}{\mathbb{H}}
\newcommand{\supp}{\mathrm{supp}}
\newcommand{\conn}[1]{\xleftrightarrow{#1}}
\newcommand{\diam}{\operatorname{diam}}

\DeclareMathOperator{\I}{{\bf{1}}}

\newcommand{\SymDiff}{\bigtriangleup{ }}
\newcommand{\Set}[1]{\{#1\}}

\newcommand{\ScoreFunc}{T}

\newcommand{\NonEmptyParticles}{{\cC^{(d)}}}
\newcommand{\BorelOf}[1]{{\cB(#1)}}
\newcommand{\BoundedBorelOf}[1]{{\cB_b(#1)}}

\newcommand{\Configs}[1]{{\bN_{#1}}}
\newcommand{\Events}[1]{{\cN_{#1}}}

\newcommand{\Proba}{\BP}
\newcommand{\Expect}{\BE}

\newcommand{\GenPP}[1]{{\Phi_{#1}}}
\newcommand{\Poisson}[1]{{\Pi_{#1}}}
\newcommand{\PercThreshold}[1]{\lambda_c(#1)}

\newcommand{\Gibbs}[1]{{\Xi_{#1}}}

\newcommand{\Thinned}[2]{{\Theta^{#1}_{#2}}}

\newcommand{\CorrFun}[1]{\rho_{#1}}
\newcommand{\ASet}{\Psi}
\newcommand{\BSet}{\Gamma}
\newcommand{\CSet}{\Upsilon}
\newcommand{\AtLeastSoManyPointsOn}[1]{E_{#1}}

\begin{document}

\title{Decorrelation of a class of Gibbs particle processes\\ and asymptotic properties of $U$-statistics}
\date{\today}
\author{
 Viktor Bene\v{s}
 \footnote{
  Charles University in Prague, Faculty of Mathematics and Physics,
  Department of Probability and Mathematical Statistics,
  Sokolovsk\'{a} 83,
  18675 Praha 8,
  Czech Republic,
  benesv@karlin.mff.cuni.cz
 },
 Christoph Hofer-Temmel
 \footnote{
  Dutch Defense Academy,
  Faculty of Military Sciences,
  Het Nieuwe Diep 8,
  1781 AC Den Helder,
  The Netherlands,
  math@temmel.me
 }
 \footnote{
 CWI,
 Science Park 123,
 1098 XG Amsterdam,
 The Netherlands
 },
 G\"unter Last
 \footnote{
  Karlsruhe Institute of Technology,
  Department of Mathematics,
  Postfach 6980,
  D-76049 Karlsruhe,
  Germany,
  guenter.last@kit.edu
 },
 Jakub Ve\v{c}e\v{r}a
 \footnote{
  Charles University in Prague,
  Faculty of Mathematics and Physics,
  Department of Probability and Mathematical Statistics,
  Sokolovsk\'{a} 83,
  18675 Praha 8,
  Czech Republic,
  vecera@karlin.mff.cuni.cz
 }
}

\maketitle

\begin{abstract}
\noindent
We study a stationary Gibbs particle process with deterministically bounded particles on Euclidean space defined in terms of an activity
parameter and non-negative interaction potentials of finite range. Using disagreement percolation we prove exponential decay of the correlation
functions, provided a dominating Boolean model is subcritical. We also prove this property
for the weighted moments of a $U$-statistic of the process.
Under the assumption of a suitable lower bound on the variance,
this implies a central limit theorem for
such $U$-statistics of the Gibbs particle process.
A byproduct of our approach is a new uniqueness result
for Gibbs particle processes.
\end{abstract}

\bigskip


\vspace{0.1cm}

\section{Introduction}
\label{sec_intro}

Starting with the seminal paper~\cite{SY13}, the limit theory for
functionals of Gibbs point processes on Euclidean space
has recently attracted a lot of attention~\cite{BYY, Tor17, XY15}.
In the present paper we derive
asymptotic moment properties
of certain Gibbs processes of geometrical objects.
 There are different background frames to deal with this problem in the literature.
The first one is to extend asymptotic results to Gibbs marked
point processes~\cite{Mase00}. In applications marks
describe the geometric properties of particles or they can be particles themselves.
The generalization of results from point processes to marked point processes is sometimes claimed as obvious in the literature (cf.~\cite[Remark 3.7]{DDG}).
However, depending on the circumstances
the details of such an extension require additional effort.
Another approach is to parametrize some particle attributes and to deal
with the point processes on the parameter space.
See~\cite{VB17} for an application of the method
of moments to a specific Gibbs model of this type.
In the present paper we have chosen a third approach dealing directly with particle
processes, defined as point processes on the space of compact
sets equipped with the Hausdorff distance as in~\cite{SW08}.

We study a
stationary Gibbs particle process $\Xi$ on $\R^d$ defined in terms of a
family of higher-order potentials with finite interaction range and an activity parameter,
assuming that the size of the particles is deterministically bounded.
Some first limit results for Gibbs particle
processes with pair potentials have been derived in~\cite{FlimmelBenes},
using Stein's method as in~\cite{Tor17}.
Let $W_n$ denote a centered cube of volume $n\in\NN$.
We are interested in the asymptotic behaviour of $U$-statistics
of the form
\begin{equation*}
 F_n:=
 \frac{1}{k!}\int h(K_1,\dotsc, K_k)\,\Xi_n^{(k)}(\dd(K_{1},\dotsc,K_{k})),
 \quad n\in\NN,
\end{equation*}
where $h$ is a symmetric and measurable function
of $k\in\NN$ particles and $\Xi_n^{(k)}$ is the restriction of
the $k$\textsuperscript{th} factorial measure of $\Xi$ to $(W_n)^k$.
For small activity parameter (and under some additional
technical assumptions) we prove exponential decay of correlations
for weighted moments of $U$-statistics.
Under an additional assumption on the asymptotic variance
this implies
a central limit theorem (CLT) for the standardized sequence $(F_n)_{n\in\NN}$.
Our main technical tools are some methods from~\cite{BYY}
and a disagreement coupling~\cite{HTH}
of two Gibbs processes with a dominating Poisson particle
process. The exponential decay of the pair correlation function via disagreement percolation property
has been proved previously in~\cite{HTH} in a comparable setting.

The paper is organized as follows.
In Section 2, we define our Gibbs model and provide some of its
basic properties. Lemma~\ref{lem_papangelou_palm} provides
the Papangelou intensity of the Palm distribution with respect to\ the
conditional Gibbs distribution for Gibbs processes with arbitrary
Papangelou intensity and is  new in this generality.
Likewise, Lemma~\ref{lem_stoch_dom_palm} on the stochastic
domination of reduced Palm distribution by a Poisson process
might also be of some independent interest.

In Section~\ref{sec_dap_decorrelation}, first the
  existence of disagreement percolation in our setting is proved in
  Theorem~\ref{thm_dat}. Then we prove the fast decay of correlations
provided that the activity is below the percolation threshold of the
associated Boolean model in Theorem~\ref{thm_exp_decorrelation}.  This
result not only strengthens and generalizes the results
in~\cite{SY13}, but also holds for a wider range of the activity
parameter.  As a byproduct we obtain with
Corollary~\ref{cor_uniqueness} a new uniqueness result.

In Section~\ref{sec_applications} we
study a $U$-statistics of order $k$ of the Gibbs particle process
in the above subcritical regime. We prove exponential decay of
correlations for weighted moments (Theorem~\ref{thm_score_factorizes})
and derive mean and variance asymptotics (Theorem~\ref{thm_expect_limit}).
Under an additional assumption on the variance asymptotics this
implies a central limit theorem (Theorem~\ref{thm_CLT}).

\section{Preliminaries}
\label{sec_preliminaries}

\subsection{Particle processes}
\label{sec_particleprocess}

Let $\R^d$ be Euclidean $d$-dimensional space with Borel
$\sigma$-field $\cB^d$,
and let $\cB^d_{b}$ denote the system of bounded Borel sets. Let
$\cC^d$ be the space of compact subsets (particles) of $\R^d$.  Let
$\NonEmptyParticles:=\cC^{d}\setminus\{\emptyset\}$ be equipped with
the Hausdorff metric $d_H$ (see~\cite{LastPenrose17,SW08}) and the associated
Borel $\sigma$-field $\cB(\NonEmptyParticles)$.
As usual for metric spaces, we define for non-empty sets $\Psi, \Gamma\subset\cC^{(d)}$
\begin{align}
\label{eq_pseudodistance}
  d(\Psi,\Gamma):=\inf_{A\in\Psi,\,B\in\Gamma}d_H(A,B).
\end{align}
To avoid confusion, our notation $d(\Psi,\Gamma)$ does not reflect the underlying
metric $d_H$. Let $\bN$ denote the space of all measures $\xi$ on
$\NonEmptyParticles$ with values in $\NN_0\cup\{\infty\}$ such that
$\xi(B(K,r))< \infty$, for each $K\in \NonEmptyParticles$ and each $r\ge 0$,
where $B(K,r):=\{L\in\NonEmptyParticles\mid{}d_H(K,L)\le r\}$
is the ball with radius $r$ centred at $K$.
As usual (see e.g.~\cite{LastPenrose17}), we equip this space with the smallest $\sigma$-field $\cN$ such
that the mappings $\xi\mapsto \xi(\Psi)$ are measurable for
each $\Psi \in \cB(\NonEmptyParticles)$.

\begin{definition}\rm
A {\em particle process} $\Xi$ in $\R^d$ is a random element of $\bN$,
defined over some fixed probability space $(\Omega,\mathcal{A},\Proba)$.
Such a particle process is said to be {\em stationary} if
$\theta_x\Xi\overset{d}{=}\Xi$, for each $x\in\R^d$, where, for each
measure $\xi$ on $\NonEmptyParticles$, we set
$\theta_x\xi:=\int \I\{K+x\in\cdot\}\,\xi(\dd K)$ with
$K+x:=\{y+x \mid{} y\in K\}$.
\end{definition}

Let $z(K)$ denote the centre of the circumscribed ball of $K \in \NonEmptyParticles$
and note that $z(K+x)=z(K)+x$ for all $(K,x)\in \NonEmptyParticles\times\R^d$.
We say that a particle process is \emph{simple}, if $\Xi(z^{-1}(x))\in\{0,1\}$,
$x\in \R^{d}$, holds almost surely. In the following, we consider only
simple stationary particle processes.
We also assume that
$\Proba(\Xi(\NonEmptyParticles )\ne 0)=1$.
The {\em intensity} $\gamma$ of such a particle process $\Xi$
is defined by
\begin{align*}
\gamma:=\BE\bigg[\int \I\{z(K)\in[0,1]^d\}\,\Xi(\dd K)\bigg].
\end{align*}
The {\em intensity measure} $\BE[\Xi]$ of $\Xi$ is the
measure $A\mapsto\BE[\Xi(A)],\;A\in \cB(\NonEmptyParticles)$.

An important example of a particle process is a {\em Poisson process}
$\Pi_{\mu}$ on $\NonEmptyParticles$, whose intensity measure $\mu$ is defined by
\begin{equation}\label{eq_measure_mu}
 \mu:=\iint \I\{K+x\in \cdot\}\,\BQ(\dd K)\,\dd x,
\end{equation}
where $\dd x$ refers to integration with respect to\ the Lebesgue measure $\mathcal{L}_d$ on $\R^d$
and $\BQ$ is some fixed probability measure on $\NonEmptyParticles$.
We refer to~\cite[Chapter 3]{LastPenrose17} for the definition and fundamental properties
of general Poisson processes. More generally, we consider the
Poisson processes $\Pi_{\lambda\mu}$ with intensity measure $\lambda\mu$,
where $\lambda >0$. Under some integrability assumptions
on $\BQ$, the Poisson process $\Pi_{\lambda\mu}$ exists as a stationary
particle process~\cite{SW08}. The number $\lambda$
is the intensity of $\Pi_{\lambda\mu}$ while $\BQ$
is called the {\em particle distribution} of $\Pi_{\lambda\mu}$.
It is no restriction of generality to assume that
\begin{equation}
\label{eq_centering}
 \BQ(\cC_{\mathbf{0}}^{(d)})=1,
\end{equation}
where $\cC_{\mathbf{0}}^{(d)}:=\{K\in \NonEmptyParticles \mid{}z(K) = \mathbf{0}\}$,
and $\mathbf{0}$ denotes the origin in $\R^d$.
However, we make the crucial assumption that there exists $R>0$ such that
\begin{equation}
\label{eq_bounded_set}
 \BQ(\{K\in \NonEmptyParticles  \mid{} K \subseteq B(\mathbf{0},R)\}) = 1,
\end{equation}
where $B(x,R)$ is the closed Euclidean ball with radius $R$ centered at $x\in\R^d$.
This is the deterministic bound on the particle size.

Given $m\in\NN$ and $\xi\in\bN$, the $m$\textsuperscript{th} \em{factorial measure} $\xi^{(m)}$
of $\xi$ is the measure on $(\NonEmptyParticles)^m$ defined by
\begin{equation*}
 \xi^{(m)}(\cdot):=
 \int \I\{(K_1,\dotsc,K_m)\in\cdot\}\text{$\I\{K_{i}\ne K_{j}$ for $i\ne j$} \}
   \,\xi^{m}(\dd(K_{1},\dots,K_{m})).
\end{equation*}
For us this is only of relevance if $\xi(\{K\})\le 1$, for each $K\in\cC^{(d)}$.
Then $\xi$ is called {\em simple}.
In this case, $\xi^{(m)}$ coincides with the standard definition of
the factorial measure~\cite[Chapter 4]{LastPenrose17}.
The $m$\textsuperscript{th} {\em factorial moment measure} $\alpha^{(m)}$ of a simple particle process
$\Xi$ is defined by $\alpha^{(m)}:=\Expect [\Xi^{(m)}]$.

\begin{definition}\label{def2.2}
Let $p\in\NN$.  The $p$\textsuperscript{th} {\em Palm distributions}
of a particle process $\Xi$ is a family
$\Proba_{K_1, \dotsc, K_p}$, $K_1, \dotsc, K_p \in \NonEmptyParticles$,
of probability measures on $\bN$ satisfying
\begin{multline}\label{eq_palm}
 \Expect\bigg[\int f(K_1,\dotsc ,K_p,\Xi)\,\Xi^{(p)}(\dd(K_1,\dotsc ,K_p))\bigg]
 \\
 =\iint f(K_1,\dotsc ,K_p,\xi )
 \,\Proba_{K_1,\dotsc,K_p}(\dd\xi)
 \,\alpha^{(p)}(\dd(K_1,\dotsc ,K_p)),
\end{multline}
for each non-negative measurable function $f$ on $(\NonEmptyParticles )^p\times\bN$.
\end{definition}

Palm distributions are well-defined whenever the $p$\textsuperscript{th} factorial
moment measure $\alpha^{(p)}$ of $\Xi$ is $\sigma$-finite. They
can be chosen such that
$(K_1,\dotsc,K_p)\mapsto \Proba_{K_1,\dotsc ,K_p}(A)$ is a measurable
function on $(\NonEmptyParticles )^p$, for each $A\in\mathcal{N}$. The
reduced Palm distribution $\Proba^{!}_{K_{1},\dotsc,K_{p}}$ of $\Xi$
is defined by means of the equality
\begin{equation}\label{reducedPalm}
 \int f(K_1,\dots ,K_p,\xi)\,\Proba^!_{K_1,\dotsc,K_p}(\dd\xi)
 =\int f(K_1,\dots ,K_p,\xi-\delta_{K_1}-\dotso-\delta_{K_p})
 \,\Proba_{K_1,\dotsc ,K_p}(\dd\xi),
\end{equation}
valid for every non-negative measurable function $f$ on
$(\cC^{(d)})^p\times\bN$.
We abuse our notation by writing, for each measurable $g\colon\bN\to\R$,
\begin{align*}
\Expect_{K_1,\dotsc ,K_p}[g(\Xi)]:=\int g(\xi)\,\Proba_{K_1,\dotsc,K_p}(\dd\xi)
\quad\text{and}\quad
\Expect^!_{K_1,\dotsc ,K_p}[g(\Xi)]:=\int g(\xi)\,\Proba^!_{K_1,\dotsc,K_p}(\dd\xi).
\end{align*}

Given $\Psi\in \cB(\NonEmptyParticles)$, we define
$\bN_{\Psi} :=\{\xi \in \bN \mid{} \xi(\Psi^{c}) = 0 \}$ and let $\cN_{\Psi}$
denote the $\sigma$-field on this set of measures.
Given $\xi\in\bN$, $B\in\cB^d$ and
$\Psi\in\cB(\NonEmptyParticles)$, we denote by $\xi_B$ and $\xi_\Psi$ the
restrictions of $\xi$ to $z^{-1}(B)$ and $\Psi$, respectively.
Finally, we set $\cB_b(\NonEmptyParticles):=\{z^{-1}(B)\mid{} B\in\cB^d_b\}$.

\subsection{Gibbs particle processes}
\label{sec_gibbsparticleprocess}

In this subsection we present some fundamental facts on Gibbs processes in a general setting.
We base our definition of a Gibbs process on the following {\em GNZ-equation},
referring e.g.\ to~\cite{Jansen19} for a discussion of the literature.

\begin{definition}\label{def_gibbs}\rm Let
$\kappa\colon \NonEmptyParticles \times\bN \to[0,\infty)$
be a measurable function and $\lambda>0$.
A particle process $\Xi$ is called a
{\em Gibbs process} with {\em Papangelou conditional intensity}
$\kappa$ and {\em activity parameter} $\lambda>0$, if
\begin{equation}\label{eq_GNZ}
 \Expect\bigg[\int f(K,\Xi-\delta_K)\,\Xi(\dd K)\bigg]
  =\lambda \Expect\bigg[ \int f(K,\Xi)\kappa(K,\Xi)\,\mu(\dd K)\bigg]
\end{equation}
holds for all measurable
$f\colon \NonEmptyParticles \times\bN \to[0,\infty)$, where
$\delta_K$ is the {\em Dirac measure} located at $K$ and
$\mu$ is given by~\eqref{eq_measure_mu}.
\end{definition}

In the following, we fix a Gibbs process with Papangelou intensity
$\kappa$ and activity $\lambda$ as in
Definition~\ref{def_gibbs}.
For $p\in\NN$, define a measurable function
$\kappa_p\colon (\NonEmptyParticles )^p\times\bN\to [0,\infty)$ by
\begin{equation}\label{eq_papan}
 \kappa_p(K_1,\dotsc,K_p,\xi)
  :=\kappa(K_1,\xi)\kappa(K_2,\xi+\delta_{K_1})
\dotsm\kappa(K_p,\xi+\delta_{K_1}+\dotsb+\delta_{K_{p-1}}).
\end{equation}
Equation~\eqref{eq_GNZ} can be iterated so as to yield
\begin{multline}
\label{eq_GNZ_papan}
 \Expect \bigg[\int f(K_1,\dotsc,K_p,\Xi-\delta_{K_1}-\dotsb-\delta_{K_p})
 \,\Xi^{(p)}(\dd(K_1,\dotsc,K_p))\bigg]
 \\
 =\lambda^p\,\Expect \bigg[\int f(K_1,\dotsc,K_p,\Xi)\kappa_p(K_1,\dotsc,K_p,\Xi)
 \,\mu^p(\dd(K_1,\dotsc, K_p))\bigg],
\end{multline}
for each measurable $f\colon (\NonEmptyParticles )^p\times\bN\to[0,\infty)$.
Therefore, $\kappa_p$ is called the conditional intensity of $p$\textsuperscript{th} order.
By~\cite[Satz 1.5]{MaWaMe79}, $\kappa_p$ is a symmetric function of the $p$ particles.

\begin{definition}\rm
Let $p\in\NN$.
The $p$\textsuperscript{th} {\em correlation function} of
a Gibbs process $\Xi$ with Papangelou intensity
$\kappa$ and activity $\lambda$ is the function
$\CorrFun{p}\colon (\NonEmptyParticles )^p\to [0,\infty]$ defined by
\begin{equation}
\label{eq_correlation}
 \CorrFun{p}(K_1,\dotsc,K_p)
  :=\lambda^p\Expect[\kappa_p(K_1,\dotsc,K_p,\Xi)].
\end{equation}
\end{definition}
Putting
\begin{equation*}
 f(K_1,\dotsc ,K_p,\xi )=\I\{K_1\in B_1,\dotsc,K_p\in B_p\},\quad
 B_1,\dotsc ,B_p\in\cB(\NonEmptyParticles),
\end{equation*}
in~\eqref{eq_GNZ_papan}, we obtain that the $p$\textsuperscript{th} factorial moment measure of $\Xi$
is given by
\begin{equation}
\label{eq_correlation_GNZ}
\alpha^{(p)}(\cdot)
=\int\I\{(K_1,\dotsc,K_p)\in\cdot\}\rho_p (K_1,\dotsc,K_p)\,\mu^p(\dd(K_1,\dotsc ,K_p)),
\end{equation}
justifying our terminology.

We define a measurable function $H\colon\bN\times\bN\to (-\infty,\infty]$, the \emph{Hamiltonian}, by
\begin{equation*}
 H(\xi,\chi):=
 \begin{cases}
 0,&\text{if $\xi(\NonEmptyParticles)=0$},
 \\
 -\log \kappa_m(K_1,\dotsc,K_m,\chi),&
 \text{if $\xi=\delta_{K_1}+\dotsb+\delta_{K_m}$ with $K_1,\dotsc,K_m\in \NonEmptyParticles$},
 \\
 \infty,&\text{if $\xi(\NonEmptyParticles)=\infty$}.
 \end{cases}
\end{equation*}
For $\Psi\in\cB(\NonEmptyParticles)$, denote by
$\Pi_{\Psi,\lambda\mu}:=(\Pi_{\lambda\mu})_\Psi$ the restriction
of the Poisson process $\Pi_{\lambda\mu}$ to $\Psi$.
The {\em partition function}  $Z_\Psi\colon\bN\to (0,\infty]$ of $\Xi$ (on $\Psi$)
is defined by
\begin{align}\label{partitionf}
 Z_\Psi(\chi):=
   \Expect \left[ e^{-H(\Pi_{\Psi,\lambda\mu},\chi)} \right],
\end{align}
For $\Psi\in \cB_b(\NonEmptyParticles)$ we have that $\mu(\Psi)<\infty$ and
hence $Z_\Psi>0$. It was shown in~\cite{MaWaMe79} that
$Z_\Psi(\Xi_{\Psi^c})<\infty$ $\BP$--a.s.\ and
that the following {\em DLR-equations}~\cite{Ruelle70,Kallenberg17,Mase00} hold:
\begin{equation}
\label{eq_dlr}
 \Expect[f(\Xi_\Psi)\mid \Xi_{\Psi^c}=\chi]
  =Z_\Psi(\chi)^{-1}
\Expect\left[f(\Pi_{\Psi,\lambda\mu})e^{-H(\Pi_{\Psi,\lambda\mu},\chi)}\right],
 \quad \Psi\in\cB_b(\NonEmptyParticles ),
\end{equation}
for $\Proba(\Xi_{\Psi^c}\in\cdot)$--a.s.\ $\chi \in \bN_{\Psi^{c}}$ and each
measurable $f\colon\bN\to[0,\infty)$.

Given $p\in \NN$, $\Psi\in\cB(\NonEmptyParticles)$ and
$\chi\in\bN_{\Psi^c}$, we write
$\Proba_{\Psi,\chi,K_1,\dotsc,K_p}^{\,!}$, $K_1,\dotsc,K_p\in\Psi$, for the reduced Palm
distribution of $\Xi_\Psi$ with boundary condition $\chi\in\bN_{\Psi^c}$.
Formally, this is the Palm distribution of the conditional distribution
$\Proba(\Xi_\Psi\in\cdot \mid \Xi_{\Psi^c}=\chi)$.
The corresponding Papangelou intensity is denoted as $\kappa_{\Psi,\chi,K_1,\dotsc,K_p}$.

\begin{lemma}\label{lem_papangelou_palm}
Let $p\in\NN$, $\Psi\in\cB_b(\NonEmptyParticles)$ and $\chi\in\bN_{\Psi^c}$.
A version of $\kappa_{\Psi,\chi,K_1,\dotsc,K_p}$
is given by
\begin{equation*}
 \kappa_{\Psi,\chi,K_1,\dotsc,K_p}(K,\xi)
 =\kappa(K,\xi+\chi+\delta_{K_1}+\dotso+\delta_{K_p}),
 \quad K_1,\dotsc,K_p\in\Psi,\,\xi\in\bN.
\end{equation*}
\end{lemma}

\begin{proof}
The proof is straightforward and given here for
completeness~\cite[page 17]{CorMoeWaa}.
Without restricting generality we assume that $\lambda=1$.

Let $g\colon\NonEmptyParticles\times\bN\to[0,\infty)$ be measurable.
By the DLR-equation and the Mecke equation for $\Pi:=\Pi_{\mu}$
(see~\cite[Theorem 4.1]{LastPenrose17}),
\begin{align*}
 \iint g(K,\xi)\,\xi(dK)&\,\Proba(\Xi_\Psi\in\dd\xi \mid \Xi_{\Psi^c}=\chi)
 \\
 &=Z_\Psi(\chi)^{-1}\BE\bigg[
  \int_{\Psi} g(K,\Pi_\Psi+\delta_K)
  e^{- H(\Pi_\Psi+\delta_K,\chi)}\,\mu(\dd K)
 \bigg].
\end{align*}
Since $H(\Pi_\Psi+\delta_K,\chi)=H(\Pi_\Psi,\chi)-\log \kappa(K,\Pi_\Psi+\chi)$
the above right-hand side equals (again by the DLR-equation)
\begin{equation*}
 \int\int_{\Psi} g(K,\xi+\delta_K)\kappa(K,\xi+\chi)\,\mu(\dd K)\,
 \Proba(\Xi_\Psi\in\dd\xi \mid \Xi_{\Psi^c}=\chi).
\end{equation*}
Hence $\Proba(\Xi_\Psi\in\cdot \mid \Xi_{\Psi^c}=\chi)$ is a Gibbs
distribution whose Papangelou kernel $\kappa_{\Psi,\chi}$ is given by
$\kappa_{\Psi,\chi}(K,\xi)=\I\{K\in\Psi\}\kappa(K,\xi+\chi)$.

By the first step of the proof it suffices to determine the
reduced Palm distributions of a Gibbs process $\Xi$.
It is convenient to write $\mathbf{K}_p:=(K_1,\dotsc,K_p)$ and
$\delta_{\mathbf{K}_p}:=\delta_{K_1}+\dotsb+\delta_{K_p}$,
for $K_1,\dotsc,K_p\in \NonEmptyParticles$.
We proceed by induction over $p$.
Let $g\colon (\NonEmptyParticles)^p\to[0,\infty)$ and
$f\colon \NonEmptyParticles \times\bN\to[0,\infty)$
be measurable. By~\eqref{reducedPalm} and~\eqref{eq_palm},
\begin{align*}
 \iiint g(\mathbf{K}_p)&f(K,\xi-\delta_K)
 \,\xi(\dd K)\,\Proba_{\mathbf{K}_p}^{\,!}(\dd \xi)
 \,\rho_p(\mathbf{K}_p)\,\mu^p(\dd \mathbf{K}_p)
 \\
 &=\BE\bigg[\iint g(\mathbf{K}_p)f(K_{p+1},\Xi-\delta_{\mathbf{K}_{p+1}})\,(\Xi-\delta_{\mathbf{K}_p})(\dd K_{p+1})\,
 \Xi^{(p)}(\dd \mathbf{K}_p)\bigg].
\end{align*}
Since $(\Xi-\delta_{\mathbf{K}_p})(\dd K_{p+1})\Xi^{(p)}(\dd \mathbf{K}_p)
 =\Xi^{(p+1)}(\dd \mathbf{K}_{p+1})$, we obtain from~\eqref{eq_GNZ_papan}
that the above right-hand side equals
\begin{align}
 \notag
 \BE&\bigg[
  \int g(\mathbf{K}_p)f(K_{p+1},\Xi)\kappa_{p+1}(\mathbf{K}_{p+1},\Xi)
  \,\mu^{p+1}(\dd \mathbf{K}_{p+1})
 \bigg]
 \\\label{eq_p_to_pplus1}
 &=\BE\bigg[
  \iint g(\mathbf{K}_p)f(K_{p+1},\Xi)
  \kappa(K_{p+1},\Xi+\delta_{\mathbf{K}_p})\kappa_{p}(\mathbf{K}_{p},\Xi)
  \,\mu^{p}(\dd \mathbf{K}_{p})\,\mu(\dd K_{p+1})
 \bigg],
\end{align}
where the identity comes from the definition of $\kappa_{p+1}$.
It follows directly from~\eqref{eq_palm} and~\eqref{eq_GNZ_papan} that
$\Proba_{\mathbf{K}_p}^{\,!}$ is for $\alpha^{(p)}$--a.e.\ $\mathbf{K}_p$
absolutely continuous with respect to the distribution of $\Xi$
with density $\kappa_p(\mathbf{K}_p,\cdot)/\rho_p(\mathbf{K}_p)$,
where $a/0:=0$ for all $a\ge 0$. Therefore, expression~\eqref{eq_p_to_pplus1} equals
\begin{equation*}
 \iiint g(\mathbf{K}_p)f(K_{p+1},\xi)\kappa(K_{p+1},\xi+\delta_{\mathbf{K}_p})
 \,\Proba_{\mathbf{K}_p}^{\,!}(\dd \xi)\,\mu(\dd K_{p+1})
 \,\rho_{p}(\mathbf{K}_{p})\mu^{p}(\dd \mathbf{K}_{p}).
\end{equation*}
This shows that the Papangelou intensity of
$\Proba_{\mathbf{K}_p}^{\,!}$ is for $\alpha^{(p)}$--a.e.\ $\mathbf{K}_p$
given by the function
$(K,\xi)\mapsto \kappa(K,\xi+\delta_{\mathbf{K}_p})$, as required.
\end{proof}

\subsection{Stochastic domination}
\label{sec_stochdom}

For $\xi,\xi'\in\bN$, we write $\xi\le\xi'$ if $\xi(\ASet)\le\xi'(\ASet)$
for all $\ASet\in\BorelOf{\NonEmptyParticles}$.
An event $E\in\cN$ is \emph{increasing}, if, for all
$\xi,\xi'\in{}\Configs{\ASet}$, $E{}\ni{}\xi\le\xi'$ implies that
$\xi'\in E$.
Another viewpoint is that $E$ is closed under the
addition of point measures.

A particle process $\Xi$ is {\em stochastically dominated} by
another particle process $\Xi'$ if
$\BP(\Xi\in E)\le \BP(\Xi'\in E)$, for each increasing $E\in\cN$.
In this case we write $\Xi\overset{d}{\le}\Xi'$
and also $\BP(\Xi\in\cdot)\overset{d}{\le}\BP(\Xi'\in\cdot)$.
By the famous Strassen theorem this implies the existence
of a coupling $(\tilde\Xi,\tilde \Xi')$ of $(\Xi,\Xi')$
such that $\tilde\Xi\le\tilde\Xi'$ almost surely.
In this context we call $\tilde\Xi$ a \emph{thinning} of $\tilde\Xi'$.
It then follows that $\Expect f(\Xi)\le \Expect f(\Xi')$ for all
{\em increasing} measurable $f\colon\bN\to[0,\infty]$,
where $f$ is called increasing if $f(\xi)\le f(\xi')$ whenever
$\xi\le\xi'$.

A classical example is the stochastic domination of
$\Poisson{\alpha\mu}$ by $\Poisson{\beta\mu}$ for
$\alpha\le{}\beta$.
Later we use the following deeper fact.

\begin{lemma}\label{lem_stoch_dom_palm}
 Suppose that $\Xi$ is a Gibbs particle
process with Papangelou intensity $\kappa\le 1$ and activity $\lambda$.
Then $\Proba(\Xi\in\cdot) \overset{d}{\le}\BP(\Pi_{\lambda\mu}\in\cdot)$.
Furthermore we have for each $p\in\NN$ that
$\Proba_{K_1,\dotsc,K_p}^{\,!}\overset{d}{\le}\BP(\Pi_{\lambda\mu}\in\cdot)$,
for $\alpha^{(p)}$--a.e.\ $(K_1,\dotsc,K_p)$,
where $\alpha^{(p)}$ is the $p$\textsuperscript{th} factorial moment measure of $\Xi$.
\end{lemma}

\begin{proof}
We only prove the second assertion. The proof of the first assertion
is simpler (and in fact a special case). Let $p\in\NN$.
We use the notation of the proof of
Lemma~\ref{lem_papangelou_palm}. By this lemma
and~\cite[Theorem 1.1]{GeorKun} used for finite Gibbs processes,
we have that
$\Proba_{\Psi,\chi,\mathbf{K}_p}^{\,!}\overset{d}{\le}\BP(\Pi_{\Psi,\lambda\mu}\in\cdot)$, for each $\Psi\in \cB_b(\NonEmptyParticles)$
for $\BP(\Xi_{\Psi^c}\in\cdot)$--a.e.\ $\chi$ and
$\alpha^{(p)}$--a.e.\ $\mathbf{K}_p\in\Psi^p$.
Hence, the definition of Palm distributions implies
for each measurable increasing $f\colon\bN\to[0,\infty)$
and each measurable
$g\colon\NonEmptyParticles\times\bN\to[0,\infty)$ $\BP$--a.s., that
\begin{align*}
\iint f(\xi-\delta_{\mathbf{K}_p})g(\mathbf{K}_p)&\,\xi^{(p)}(\dd \mathbf{K}_p)\,
\BP(\Xi_\Psi\in\dd \xi \mid \Xi_{\Psi^c})\\
&\le \iint f(\xi)g(\mathbf{K}_p)\,\BP(\Pi_{\Psi,\lambda\mu}\in\dd\xi)\,
\Expect[(\Xi_{\Psi})^{(p)}\in\dd \mathbf{K}_p \mid \Xi_{\Psi^c}].
\end{align*}
Taking expectations yields
\begin{align*}
\BE\bigg[\int_{\Psi^p} f((\Xi-\delta_{\mathbf{K}_p})_\Psi)g(\mathbf{K}_p)\,\Xi^{(p)}(\dd \mathbf{K}_p)\bigg]
\le
\int_{\Psi^p} \Expect[f(\Pi_{\Psi,\lambda\mu})]g(\mathbf{K}_p)\,\alpha^{(p)}(\dd \mathbf{K}_p),
\end{align*}
so that $\Proba_{\mathbf{K}_p}^{\,!}(\{\xi\in\bN\mid\xi_\Psi\in\cdot\})
\overset{d}{\le}\BP(\Pi_{\Psi,\lambda\mu}\in\cdot)$
for $\alpha^{(p)}$--a.e.\ $\mathbf{K}_p\in\Psi'$, for each measurable $\Psi'\subset\Psi$.
Exactly as in the proof of~\cite[Corollary 3.4]{GeorYoo05}, we let
$\Psi\uparrow\NonEmptyParticles$ to obtain that
$\Proba_{\mathbf{K}_p}^{\,!}\overset{d}{\le}\BP(\Pi_{\lambda\mu}\in\cdot)$
for $\alpha^{(p)}$--a.e.\ $\mathbf{K}_p\in\Psi'$ and hence the assertion.
\end{proof}
\begin{remark}\rm
The definitions and results of Subsections 2.1 and 2.2
apply to Gibbs processes on a general complete separable
metric space equipped with a locally finite measure $\mu$.
\end{remark}

\subsection{Admissible Gibbs particle processes}

A family $\varphi:=(\varphi_n)_{n\ge 2}$ of \emph{higher-order
  interaction potentials} consists of measurable, symmetric and
translation-invariant functions
$\varphi_n:(\cC^{d})^n\to(-\infty,\infty]$.  The potentials have
\emph{finite interaction range} $R_{\varphi}$, if
$\varphi_n(K_1,\dotsc,K_n)=0$, for every $n\ge 2$ and all
$K_1\dotsc,K_n\in\cC^{d}$ with
$\max\{d_H(K_i,K_j)\mid{}1\le{}i<j\le n\}>R_{\varphi}$.

Define the Papangelou intensity
$\kappa\colon \NonEmptyParticles \times\bN \to[0,\infty)$ by  $\kappa(K,\xi):=0$,
if $K\in\operatorname{supp}\xi$, and otherwise
\begin{equation}
\label{eq_papangelou_intensity}
 \kappa(K,\xi):=
 \exp\left[
   -\sum_{n=2}^\infty \frac{1}{(n-1)!}
   \int \varphi_n(K,L_1,\dotsc,L_{n-1}) \,\xi^{(n-1)}(\dd(L_1,\dotsc,L_{n-1}))
 \right].
\end{equation}
The function $\kappa$ is
measurable and translation-invariant.
Here and later we make the following convention
regarding the series in the exponent
of~\eqref{eq_papangelou_intensity}. If the sum over the negative terms diverges, then the whole series is set to
zero. We assume that this is not the case for all $\xi\in\bN$
and $\mu$--a.e.\ $K$. We also assume that $\kappa\le 1$.  While
individual potentials might be attractive (i.e., negative), their
cumulative effect must be repulsive (i.e., non-negative).

Proving the existence of a Gibbs process with a given Papangelou
intensity is a non-trivial task.  The literature contains many
existence results under varying assumptions of
generality~\cite{DDG,FlimmelBenes,Mase00,Ruelle70,SY13}, none of which
seems to cover our current setting.  For our main findings (for
instance Theorems~\ref{thm_exp_decorrelation} and~\ref{thm_CLT}) we
need to restrict the range of the activity parameter to a finite
interval, the subcritical percolation regime of the associated Poisson
particle process, c.f. Section~\ref{sec_perc}.  In that case, the
Gibbs distribution is not only uniquely determined (see
Corollary~\ref{cor_uniqueness}) but can be expected to exist.
In the case of a non-negative pair potential we have the
following result.

\begin{remark}\rm Assume that the Papangelou intensity $\kappa$
is given by a non-negative pair potential, i.e., assume
that $\kappa$ is given by~\eqref{eq_papangelou_intensity} with $n=2$.
Assume also that $\varphi_2$ has a finite interaction range,
or, more generally that $\int \big(1-e^{-\varphi_2(K,L)}\big)\,\mu(\dd K)<\infty$
for $\mu$--a.e.\ $L$. (By assumption~\eqref{eq_bounded_set} we have that
$\mu(\Psi)<\infty$ for any ball $\Psi\subset\NonEmptyParticles$.)
Under these assumptions it
has been shown in~\cite{Jansen19} that a Gibbs process exists.
\end{remark}

We do not further address the existence problem in this paper
and proceed under the assumption that the Gibbs process exists.

For all $\xi,\chi\in\bN$ with disjoint supports, the Hamiltonian $H$
takes the form
\begin{multline*}
 H(\xi,\chi):=
  \sum_{n=2}^\infty
   \sum_{k=1}^n
   \frac{1}{k!(n-k)!}
    \iint \varphi_n(K_1,\dotsc,K_k,L_1,\dotsc,L_{n-k})
  \\\times
     \,\xi^{(k)}(\dd(K_1,\dotsc,K_k))
     \,\chi^{(n-k)}(\dd(L_1,\dotsc,L_{n-k})),
\end{multline*}
provided that $\xi$ is finite.
The assumption $\kappa\le 1$ implies that $H\ge 0$.
If assumptions~\eqref{eq_centering} and~\eqref{eq_bounded_set} hold, then
\eqref{eq_dlr} shows that the Gibbs process $\Xi$ has bounded particles, that is
\begin{equation*}
 \int\I\{K\not\subseteq B(z(K),R)\}\,\Xi(\dd K)=0,
 \quad \Proba\text{--a.s.}.
\end{equation*}

For clarity and to avoid lengthy formulations we make the following definition.

\begin{definition}\rm
\label{def_admissible_gibbs}
Assume that $\varphi$ is a family of higher-order potentials with finite interaction range $R_\varphi$.
Define $\kappa$ by~\eqref{eq_papangelou_intensity} and assume that $\kappa\le{}1$.
Assume also that $\BQ$ is a probability measure on $\NonEmptyParticles$
satisfying~\eqref{eq_centering} and~\eqref{eq_bounded_set}.
Let $\lambda>0$ be given.
Assume that $\Xi$ is a Gibbs particle process as in Definition~\ref{def_gibbs}, where
$\mu$ is defined by~\eqref{eq_measure_mu}.
Then, we call $\Xi$ an {\em admissible Gibbs process}.
\end{definition}

\noindent For an admissible Gibbs particle process it follows
from~\eqref{eq_papangelou_intensity},~\eqref{eq_papan} and~\eqref{eq_GNZ_papan} that
\begin{equation*}
  \CorrFun{p}(K_1,\dotsc,K_p)\le \lambda^p,\quad K_1,\dotsc,K_p\in\NonEmptyParticles.
\end{equation*}

A classic setup of a repulsive intersection-based pair potential
arises from a measurable translation invariant function
$U\colon \cC^{d}\to[0,\infty]$ with $U(\emptyset)=0$ and setting
$\varphi_2(K,L):=U(K\cap{}L)$ and $\varphi_n:=0$, for $n\ge 3$.
Assumption~\eqref{eq_bounded_set} implies an interaction range of at
most $4R$.

\begin{example}\rm
\label{ex_facet1}
For $d\geq 2$, let $\mathcal{G}_d$ be the space of $(d-1)$-dimensional
linear subspaces of $\R^d$.
Let $R>0$ and let the measure $\BQ$ be concentrated on
\begin{equation*}
 V:=\{A\cap B({\bf 0},R)\mid{} A\in\mathcal{G}_d\},
\end{equation*}
Then~\eqref{eq_bounded_set} holds.
The particles are called facets and $\BQ$ can be interpreted
as the distribution of their normal directions.
The space of facets is
\begin{equation}
\label{fasety}
 \tilde{V}:=\{B+x\mid{}B\in V,\;x\in\R^d\}.
\end{equation}
Let $\BH^m$ be the Hausdorff measure of order $m \in \{1,\dotsc,d\}$ on $\R^d$.
For $j\in\{1,\dotsc ,d\}$ and $K_1,\dots ,K_j\in\tilde{V}$ define, setting $0\cdot\infty=0$,
\begin{equation}
\label{eq_potq}
 Q_j(K_1,\dots ,K_j):=
 \BH^{d-j} \left(\bigcap_{i=1}^{j} K_i\right)
 \I\left\{\BH^{d-j}\left(\bigcap_{i=1}^{j} K_i\right)<\infty\right\}
 .
\end{equation}
A family $\varphi:=(\varphi_j)_{j\geq 2}$ of higher-order potentials is defined by
\begin{equation*}
 \varphi_j(K_1,\dots ,K_j):=
 a_jQ_j(K_1,\dots ,K_j),\;K_1,\dots ,K_j\in\tilde{V}
 ,
\end{equation*}
for $j\in\{2,\dotsc,d\}$, and $\varphi_j:=0$
  otherwise, where $a_2,\ldots,a_d\geq 0$ are given parameters. All these potentials
  have the finite range $R_\varphi=2R.$ The corresponding Gibbs
  particle process $\Xi$ is called the Gibbs facet process.  It is
  admissible and its existence follows from~\cite[Remarks 3.7 and
  3.1]{DDG}.  For $j\in\{2,\dotsc,d\}$, the $j$\textsuperscript{th}
  submodel is the special case of only the $j$\textsuperscript{th}
  potential being active, i.e., only $a_j>0$, and we denote it by
  $^j\Xi$.
\end{example}

\section{Disagreement percolation and moment decorrelation}
\label{sec_dap_decorrelation}

In the first subsection of this section we discuss some percolation properties of
a Poisson process. In the remaining two subsections
we fix an admissible Gibbs process
and discuss disagreement percolation and prove decorrelation
of moments in a subcritical regime.

\subsection{Percolation}
\label{sec_perc}

Define a symmetric relation on $\NonEmptyParticles$ by
setting $K\sim{}L$, if and only if, $K\cap{}L\not=\emptyset$.
For $\xi\in\Configs{}$, this defines a graph
$(\supp{}\,\xi,\sim)$.  For $K,L\in\NonEmptyParticles$, we
say that $\xi$ \emph{connects} $K$ and $L$, if there exists a finite path
between $K$ and $L$ in the graph on $\xi+\delta_K+\delta_L$. For disjoint
$\ASet,\BSet\in\BorelOf{\NonEmptyParticles}$, we say that $\xi$
connects $\ASet$ and $\BSet$, if there exist $K\in\ASet$ and
$L\in\BSet$ such that $\xi$ connects $K$ and $L$.  We write
$\ASet\conn{\xi}\BSet$ for this.

We say that $\xi$ \emph{percolates}, if its graph contains an infinite
connected component.
Because connectedness is an increasing event in $\Events{}$, there is a critical
\emph{percolation intensity}
$$
 \PercThreshold{d} \equiv\PercThreshold{d,\BQ}\in[0,\infty]
$$
for percolation of $\Poisson{\lambda\mu}$; see~\cite{MeesterRoy96}.
The following consequence of
a result in~\cite{Ziesche18} is of crucial importance for our
main results.

\begin{lemma}
\label{lem_exp_decay_con_single}
For $\lambda<\PercThreshold{d}$ and $\ASet,\BSet\in\BoundedBorelOf{\NonEmptyParticles}$ with $\ASet\subseteq{}\BSet$, there exist a monotone increasing $C_1:[0,\infty)\to[0,\infty)$  and $C_2\in\,(0,\infty)$ such that
\begin{equation}
\label{eq_exp_decay_con_single}
 \Proba\Bigl(\ASet\conn{\Poisson{\lambda\mu}}\BSet^c\Bigr)
 \le
 C_1(\diam(\ASet))\exp(-C_2d(\ASet,\BSet^c)).
\end{equation}
\end{lemma}

\begin{proof}
Let $\cC_{\mathbf{0}}^{(d,R)}:=\{K\in\cC_{\mathbf{0}}^{(d)}\mid{}K\subseteq{}B(\mathbf{0},R)\}$.
By~\eqref{eq_bounded_set} we have that $\BQ(\cC_{\mathbf{0}}^{(d,R)}) = 1$.
A slightly weakened form of the bound~\cite[Equation~(3.7)]{Ziesche18} in our notation is as follows.
For $\lambda<\PercThreshold{d}$, there exists a constant $C\in\,(0,\infty)$, such that, for all $r\ge 0$,
\begin{equation}
\label{eq_exp_decay_con_ziesche}
 \int_{\cC_{\mathbf{0}}^{(d)}}
 \Proba\Bigl(
  K\conn{\Poisson{\lambda\mu}+\delta_K}
  B(\mathbf{0},r)^c\times\cC_{\mathbf{0}}^{(d,R)}
 \Bigr)
 \,\BQ(\dd{}K)
 \le e^{-C(r-2R)}
 .
\end{equation}
The weakening is a result of a switch from the notion of connecting
two sets in $\R^d$ used in~\cite{Ziesche18} to the one connecting two
sets in $\R^d\times\cC_{\mathbf{0}}^{(d,R)}$ used above.  Thus, we correct
twice by $R$ in the exponent on the right-hand side
of~\eqref{eq_exp_decay_con_ziesche} to account for the maximum size of
$K$ and of particles in $B(0,r)^c\times\cC_{\mathbf{0}}^{(d,R)}$
respectively.

If $\ASet=\emptyset$, then the probability is zero anyway.
Proceed by assuming that $\ASet\not=\emptyset$.

Let $D:=\diam(\ASet)$.
Let $s:=d(\ASet,\BSet^c) - D - 10R$.
If $s\le{}0$, then choosing $C_1(D)\ge{}e^{C(D+10R)}$ finishes the proof.
From here on, assume that $s>0$ and that all particles are deterministically bounded by $R$.

Choose and fix $K_\ASet\in\ASet$ and let $x:=z(K_\ASet)$.
Let $\CSet:=A\times\cC_{\mathbf{0}}^{(d,R)}$ be the particles with centres of circumscribed balls within the annulus $A:=B(x,D+6R)\setminus{}B(x,D+4R)$.

For $K\in{}B(x,D+4R)\times\cC_{\mathbf{0}}^{(d,R)}=:\CSet^-$ and
$L\in{}B(x,D+6R)^c\times\cC_{\mathbf{0}}^{(d,R)}=:\CSet^+$, we have
$|z(K)-z(L)|\ge 2R$, whence $K\cap{}L=\emptyset$.  For $\xi\in\bN$,
let $P(\xi)$ be the particles of $\xi_{\CSet}$ connected in
$\xi_{\BSet\setminus{}\CSet^{-}}$ to $\BSet^c$.  Then
$\ASet\conn{\xi}\BSet^c$ implies that $P(\xi)\not=\emptyset$, because
particles in $\CSet^-$ do not intersect particles in $\CSet^+$ and
$\xi$ needs to contain at least one particle in $\CSet$.  Together
with a first moment bound, this yields
\begin{equation*}
 \Proba\Bigl(\ASet\conn{\Poisson{\lambda\mu}}\BSet^c\Bigr)
 \le{}
 \Proba(P(\Poisson{\lambda\mu})\not=\emptyset)
 \le{}
 \Expect|P(\Poisson{\lambda\mu})|
 .
\end{equation*}
We apply the Mecke equation to rewrite
\begin{equation*}
 \Expect|P(\Poisson{\lambda\mu})|
 ={}
 \Expect\int\I\{K\in{}P(\Poisson{\lambda\mu})\}\,\Poisson{\lambda\mu}(\dd{}K)
 ={}
 \int_{\CSet} \Proba\Bigl(K\in{}P(\Poisson{\lambda\mu}+\delta_K)\Bigr)
 \,\lambda\mu(\dd{}K)
 .
\end{equation*}

For $K\in\CSet$ and $L\in\BSet^c$ we have that
\begin{equation*}
 d_H(K,L)
 \ge d_H(K_\ASet,L) - d_H(K_\ASet,K)
 \ge d(\ASet,\BSet^c) -  (D + 6R + 2R)
 = s + 2R
 .
\end{equation*}
In particular, we have $d_H(K,L)\ge 2R$, so that $K\cap L=\emptyset$.
Moreover, for each $\xi\in\bN$ 
satisfying $K\conn{\xi}\BSet^c$, we obtain
that $K\conn{\xi}B(z(K),s)^c\times\cC_{\mathbf{0}}^{(d,R)}$.
This implies that
\begin{align*}
 \Proba\Bigl(K\in{}P(\Poisson{\lambda\mu}+\delta_K)\Bigr)
 \le{}
 \Proba\Bigl(K\conn{\Poisson{\lambda\mu}+\delta_K}\BSet^c\Bigr)
 \le{}
 \Proba\Bigl(
   K\conn{\Poisson{\lambda\mu}+\delta_K}
   B(z(K),s)^c\times\cC_{\mathbf{0}}^{(d,R)}
 \Bigr)
 .
\end{align*}

Combining these upper bounds and using the definition
\eqref{eq_measure_mu} of $\mu$ we see that
\begin{align*}
 \Proba\Bigl(\ASet\conn{\Poisson{\lambda\mu}}\BSet^c\Bigr)
 &\le{}
 \int_{\CSet}
 \Proba\Bigl(
   K\conn{\Poisson{\lambda\mu}+\delta_K}
   B(z(K),s)^c\times\cC_{\mathbf{0}}^{(d,R)}
 \Bigr)
 \,\mu(\dd{}K)
 \\
 &=
 \int_{\cC_{\mathbf{0}}^{(d)}}\int_A
 \Proba\Bigl(
   K+x\conn{\Poisson{\lambda\mu}+\delta_{K+x}}
   B(x,s)^c\times\cC_{\mathbf{0}}^{(d,R)}
 \Bigr)\,dx
 \,\BQ(\dd{}K).
\end{align*}
By stationarity of $\Poisson{\lambda\mu}$ this equals
\begin{align*}
 \lambda\mathcal{L}_d(A)
 \int_{\cC_{\mathbf{0}}^{(d)}}
 \Proba\Bigl(
   K\conn{\Poisson{\lambda\mu}+\delta_K}
   B(\mathbf{0},s)^c\times\cC_{\mathbf{0}}^{(d,R)}
 \Bigr)
 \,\BQ(\dd{}K)
\le
 \lambda\mathcal{L}_d(A)
 e^{-C (s-2R)},
\end{align*}
where the inequality comes from~\eqref{eq_exp_decay_con_ziesche}.
Choosing $C_1(D)\ge\lambda \mathcal{L}_d(A) e^{C(D+12R)}$ concludes the argument.
\end{proof}

Monotonicity in the particle shapes allows to control the percolation
threshold.  In the special case of $\BQ = \delta_{B(\mathbf{0},R)}$,
the measure $\mu$ becomes $\mu_R := \int \I\{B(x,R)\in \cdot\}\dd x$.
Assumption~\eqref{eq_bounded_set} implies for each $\lambda>0$
that $\Pi_{\lambda\mu}$--a.e.\ $\xi$
fulfils
\begin{equation*}
 \bigcup_{K\in\xi} K
 \subseteq \bigcup_{K\in\xi} B(z(K),R).
\end{equation*}
Hence, we can couple $\Poisson{\lambda\mu}$ and $\Poisson{\lambda\mu_R}$ such that
\begin{equation*}
 \Proba\Biggl(
  \bigcup_{K\in\Poisson{\lambda\mu}} K
  \subseteq
  \bigcup_{B\in\Poisson{\lambda\mu_R}} B
 \Biggr)=1.
\end{equation*}
A well known lower bound
(see~\cite {Penrose96}  and~\cite[Section 3.9]{MeesterRoy96})
is
\begin{equation}
\label{eq_perc_threshold_inequalities}
  \PercThreshold{d,\BQ}
  \ge
  \PercThreshold{d,\delta_{B(\mathbf{0},R)}}
  \ge
  \frac{1}{v_d 2^d R^d},
\end{equation}
where $v_d$ is the volume of the $d$-dimensional unit ball.

\subsection{Disagreement percolation}
\label{sec_dap}

For $\xi,\xi'\in{}\Configs{}$, we write $\xi\SymDiff\xi'$ for the
\emph{absolute difference measure}
$\max\Set{\xi,\xi'}-\min\Set{\xi,\xi'}$, equivalent to
$|\xi-\xi'|$. In the relevant case of $\xi$ and $\xi'$ both being
simple, there is a simpler geometric interpretation of the also simple
$\xi\SymDiff\xi'$.  Switching to the support of a simple point
measure, we see that
$\supp(\xi\SymDiff\xi')=(\supp\,\xi)\SymDiff(\supp\,\xi')$, which
motivates this overloading of the set difference operator $\SymDiff$
to point measures.

The space $(\NonEmptyParticles,d_H)$ is a complete and separable
metric space. By~\cite[Theorem 13.1.1]{Dudley_RAP}, the spaces
$(\NonEmptyParticles,\BorelOf{\NonEmptyParticles})$ and $\R$ equipped
with the Borel $\sigma$-algebra are Borel isomorphic.  That is, there
exists a measurable bijection from $\NonEmptyParticles$ to $\R$ with
measurable inverse.  We use this bijection to pull back the total
order from $\R$ to $\NonEmptyParticles$ and denote it by $\prec$.
Hence, intervals with respect to $\prec$ are in
$\BorelOf{\NonEmptyParticles}$.

For the remainder of the section we fix an admissible Gibbs
process $\Xi$ as in Definition~\ref{def_admissible_gibbs}.

\begin{theorem}\label{thm_dat}
For all $\ASet\in\BoundedBorelOf{\NonEmptyParticles}$ and
$\chi_1,\chi_2\in\Configs{\ASet^c}$, there exists a simultaneous
thinning from $\Poisson{\ASet,\lambda\mu}$ to two particle processes
$\Thinned{\ASet,\chi_1}{}$ and $\Thinned{\ASet,\chi_2}{}$ such that
$\Thinned{\ASet,\chi_i}{}$ has the distribution
$\BP(\Gibbs{\ASet}\in\cdot\mid{}\Gibbs{\ASet^c}=\chi_i)$, for $i\in\Set{1,2}$, and,
$\Proba$--a.s.,
\begin{equation}
\label{eq_dac_disagreement}
 \forall{}K\in\supp
 \Bigl(\Thinned{\ASet,\chi_1}{}\SymDiff\Thinned{\ASet,\chi_2}{}\Bigr)\colon{}
  \Set{K}
  \conn{\Thinned{\ASet,\chi_1}{}\SymDiff\Thinned{\ASet,\chi_2}{}}
  \chi_1\SymDiff\chi_2.
\end{equation}
\end{theorem}

\begin{proof}
Using $\prec$ restricted to $\ASet$ in place of the measurable ordering of a bounded Borel subset in~\cite[Section 4.1]{HTH}, and the DLR-equations as formulated in~\eqref{eq_dlr}, this theorem becomes a literal copy of the construction leading to~\cite[Theorem 3.3]{HTH}.
\end{proof}

The term disagreement percolation comes from the fact that in the
subcritical percolation regime of $\Poisson{\lambda\mu}$, there is control of a disagreement
cluster by a percolation cluster. The
finiteness of the percolation clusters guarantees uniqueness of the
Gibbs process.

\begin{corollary}\label{cor_uniqueness}
If $\lambda<\PercThreshold{d}$, then the distribution of $\Gibbs{}$ is uniquely determined.
\end{corollary}

\begin{proof}
The proof generalises in a straightforward way the proof of~\cite[Theorem 3.2]{HTH} and Theorem~\ref{thm_dat}, with the only change being that the interaction range and particle size are here two separate parameters. Because of the deterministic bound $R$ from~\eqref{eq_bounded_set} on the particle size and the finiteness of the interaction range, the arguments remain the same.
\end{proof}

\subsection{Decorrelation of moments}
\label{sec_decorrelation}

With the following theorem we establish the particle counterpart
of fast decay of correlations in~\cite[Definition 1.1]{BYY}
in the subcritical regime.

\begin{theorem}\label{thm_exp_decorrelation}
Assume that $\lambda<\PercThreshold{d}$.
There exist $c_1,c_2\in(0,\infty)$ such that Gibbs process $\Gibbs{}$ satisfies,
for all $p,q\in\NN$ and $\alpha^{(p+q)}$--a.e.\ $(K_1,\dots ,K_{p+q})$,
\begin{align}
\label{eq_exp_decorrelation}
\notag
 |\CorrFun{p+q}(K_1,\dotsc ,K_{p+q})&-
  \CorrFun{p}(K_1,\dotsc ,K_p)
  \CorrFun{q}(K_{p+1},\dotsc ,K_{p+q})|\\
 &\le \lambda^{p+q} \min(p,q)\,c_1
   \exp(-c_2 d(\{K_1,\dotsc ,K_{p}\},\{K_{p+1},\dotsc ,K_{p+q}\}))
 .
\end{align}
\end{theorem}

Combining Theorem~\ref{thm_exp_decorrelation} with
the bound~\eqref{eq_perc_threshold_inequalities} on the
percolation threshold gives
the following constraint on the activity as sufficient condition for
the exponential decay of correlations~\eqref{eq_exp_decorrelation}:
\begin{equation}
\label{eq_bound_new}
 \lambda <\frac{1}{v_d 2^d R^d}.
\end{equation}

\begin{remark}\label{r3.5}\rm
To compare our results with Proposition 2.1 in~\cite{SY13}
we consider {\em hard spheres in equilibrium}, that is we assume that
$\BQ = \delta_{B(\mathbf{0},R)}$,
$\varphi_2(K,L)=\infty\cdot\I\{K\cap L\ne \emptyset\}$ and $\varphi_n\equiv 0$
for $n\ge 3$. Then $\Xi$ can be identified with a point process on $\R^d$.
In the language from~\cite{SY13} we have that $r^\Psi=2R$ and $m^\Psi_0=0$.
Proposition 2.1 in~\cite{SY13} then shows exponential
decay of correlations whenever
\begin{align}\label{eq_bound_sy13}
 \lambda<\frac{1}{v_d(1+2R)^d}.
\end{align}
This is comparable with~\eqref{eq_bound_new}.

Our Theorem~\ref{thm_exp_decorrelation} shows exponential decay of
correlations for a broader range of activities. In fact, it is known
from simulations that in low dimensions $\lambda_c(d)$ is considerably
larger than the right-hand side of~\eqref{eq_bound_new}. (As $d\to\infty$ we have
$\lambda_c(d)v_d2^dR^d\to 1$; see~\cite{Penrose96}.) We do not see
how the methods from~\cite{SY13} can be used to prove
Theorem~\ref{thm_exp_decorrelation}.
\end{remark}

\begin{remark}\rm
In the recent paper~\cite{Jansen19} uniqueness of the distribution
of a Gibbs process on $\R^d$
has been proved with very different methods, based on a fixed point
argument for correlation functions.
It is an interesting problem
to explore the relationship between this method and disagreement
percolation. Some first answers are given in~\cite{Jansen19}.
\end{remark}

The proof of Theorem~\ref{thm_exp_decorrelation} is based on the
following lemmas.  For these lemmas and the proof of
Theorem~\ref{thm_exp_decorrelation}, fix $\lambda<\PercThreshold{d}$
and let $C_1$ and $C_2$ as in Lemma~\ref{lem_exp_decay_con_single}.

\begin{lemma}
\label{lem_increasing}
Let $p\in\NN$ and let $\ASet_1,\dotsc,\ASet_p\in\BoundedBorelOf{\NonEmptyParticles}$
be disjoint.
Let $\BSet\in\BoundedBorelOf{\NonEmptyParticles}$ with
$\bigcup_{i=1}^p\ASet_i=:\ASet\subseteq{}\BSet$.
For $\chi\in\Configs{\BSet^c}$ and increasing $E\in\cN$,
\begin{multline}
 \label{eq_increasing_spec_state}
 \Bigl|
  \Proba(\Gibbs{\ASet}\in{}E\mid{}
         \Gibbs{\BSet^c}=\chi)
 -\Proba(\Gibbs{\ASet}\in{}E)
 \Bigr|
 \\
 \le
  \Proba(\Poisson{\ASet,\lambda\mu}\in{}E)
  \sum_{i=1}^p C_1(\diam(\ASet_i))\exp(-C_2 d(\ASet,\BSet^c))
 .
\end{multline}
\end{lemma}

\begin{proof}
First, we show that, for $\chi_1,\chi_2\in\Configs{\BSet^c}$,
\begin{multline}
\label{eq_increasing_spec_spec}
 \Bigl|
  \Proba(\Gibbs{\ASet}\in{}E\mid{}
         \Gibbs{\BSet^c}=\chi_1)
  -
  \Proba(\Gibbs{\ASet}\in{}E\mid{}
         \Gibbs{\BSet^c}=\chi_2)
 \Bigr|
 \\
 \le
   \Proba(\Poisson{\ASet,\lambda\mu}\in{}E)
   \sum_{i=1}^p C_1(\diam(\ASet_i))\exp(-C_2 d(\ASet,\BSet^c))
 .
\end{multline}
By Theorem~\ref{thm_dat},
\begin{align*}
\Bigl|\Proba(\Gibbs{\ASet}\in{}E\mid{}\Gibbs{\BSet^c}=\chi_1)
&-\Proba(\Gibbs{\ASet}\in{}E\mid{}\Gibbs{\BSet^c}=\chi_2) \Bigr|\\
&= \Bigl|\Proba(\Thinned{\BSet,\chi_1}{\ASet}\in{}E)
-\Proba(\Thinned{\BSet,\chi_2}{\ASet}\in{}E)\Bigr|\\
&=  \Bigl|
   \Proba(\Thinned{\BSet,\chi_1}{\ASet}\in{}E
         ,\Thinned{\BSet,\chi_2}{\ASet}\not\in{}E)
- \Proba(\Thinned{\BSet,\chi_1}{\ASet}\not\in{}E
         ,\Thinned{\BSet,\chi_2}{\ASet}\in{}E)\Bigr| \\
&\le  \max\Set{
   \Proba(\Thinned{\BSet,\chi_1}{\ASet}\in{}E
         ,\Thinned{\BSet,\chi_2}{\ASet}\not\in{}E)
  ,\Proba(\Thinned{\BSet,\chi_1}{\ASet}\not\in{}E
         ,\Thinned{\BSet,\chi_2}{\ASet}\in{}E)
  }.
\end{align*}
By symmetry we only need to bound the first term in the above maximum.
It follows from~\eqref{eq_dac_disagreement} that
\begin{align*}
 \Proba(\Thinned{\BSet,\chi_1}{\ASet}\in{}E &
       ,\Thinned{\BSet,\chi_2}{\ASet}\not\in{}E)\\
&= \Proba(\Thinned{\BSet,\chi_1}{\ASet}\in{}E
  ,\Thinned{\BSet,\chi_2}{\ASet}\not\in{}E
  ,\emptyset\not=
   \Thinned{\BSet,\chi_1}{\ASet}\SymDiff\Thinned{\BSet,\chi_2}{\ASet} ) \\
&= \Proba(
   \Thinned{\BSet,\chi_1}{\ASet}\in{}E
  ,\Thinned{\BSet,\chi_2}{\ASet}\not\in{}E
  ,\emptyset\not=
   \Thinned{\BSet,\chi_1}{\ASet}\SymDiff\Thinned{\BSet,\chi_2}{\ASet}
   \conn{
    \Thinned{\BSet,\chi_1}{}
    \SymDiff
    \Thinned{\BSet,\chi_2}{}
   }
   \chi_1\SymDiff\chi_2 ).
\end{align*}
Since
\begin{align*}
\Big\{\Thinned{\BSet,\chi_1}{\ASet}\SymDiff\Thinned{\BSet,\chi_2}{\ASet}
\conn{\Thinned{\BSet,\chi_1}{}
\SymDiff
    \Thinned{\BSet,\chi_2}{}} \chi_1\SymDiff\chi_2\Big\}
\subseteq \Big\{\ASet\conn{\Thinned{\BSet,\chi_1}{}
    \SymDiff
    \Thinned{\BSet,\chi_2}{}} \chi_1\SymDiff\chi_2\Big\}
\subseteq
\Big\{\ASet\conn{\Poisson{\lambda\mu}}\BSet^c\Big\},
\end{align*}
we obtain that
\begin{align*}
 \Proba(\Thinned{\BSet,\chi_1}{\ASet}\in{}E,\Thinned{\BSet,\chi_2}{\ASet}\not\in{}E)
&\le  \Proba(\Poisson{\ASet,\lambda\mu}\in{}E
       ,\ASet\conn{\Poisson{\lambda\mu}}\BSet^c)
\\
&= \Proba(\Poisson{\ASet,\lambda\mu}\in{}E)
   \Proba(\ASet\conn{\Poisson{\lambda\mu}}\BSet^c)
\\
&\le \Proba(\Poisson{\ASet,\lambda\mu}\in{}E)
   \sum_{i=1}^p \Proba(\ASet_i\conn{\Poisson{\lambda\mu}}\BSet^c)
\\
&\le \Proba(\Poisson{\ASet,\lambda\mu}\in{}E)
 \sum_{i=1}^p C_1(\diam(\ASet_i))\exp(-C_2 d(\ASet,\BSet^c))
 ,
\end{align*}
where the equality results from the complete independence of a Poisson
process, the second inequality is a Boolean bound and the final equality uses~\eqref{eq_exp_decay_con_single} and the fact that $d(\ASet,\BSet^c)\le d(\ASet_i,\BSet^c)$.
This proves~\eqref{eq_increasing_spec_spec}.

To prove~\eqref{eq_increasing_spec_state}, we use the DLR-equation~\eqref{eq_dlr} to obtain that
\begin{align*}
\Bigl|  \Proba(\Gibbs{\ASet}\in{}E\mid{}\Gibbs{\BSet^c}=\chi)
&-\Proba(\Gibbs{\ASet}\in{}E)\Bigr|\\
&\le\int\Bigl|\Proba(\Gibbs{\ASet}\in{}E\mid{}\Gibbs{\BSet^c}=\chi)
-\Proba(\Gibbs{\ASet}\in{}E\mid{}\Gibbs{\BSet^c}=\chi')\Bigr|
\,\Proba(\Gibbs{\BSet^c}\in\dd{}\chi').
\end{align*}
An application of~\eqref{eq_increasing_spec_spec} to the integrand shows~\eqref{eq_increasing_spec_state}.
\end{proof}

For $\ASet\in\BoundedBorelOf{\NonEmptyParticles}$ and $n\in\NN$, let
$\AtLeastSoManyPointsOn{\ASet,n}:=\Set{\xi\in\bN\mid{}\xi(\ASet)\ge{}n}$.
The event $\AtLeastSoManyPointsOn{\ASet,n}$ is increasing.
Fix $p\in\NN$ and disjoint $\ASet_1,\dotsc,\ASet_p\in\BoundedBorelOf{\NonEmptyParticles}$.
Let $\GenPP{}$ be a particle process. Then,
\begin{equation}
\label{eq_moment_sum_probas}
 \Expect\Bigl(\prod_{i=1}^{p} \GenPP{}(\ASet_i)\Bigr)
 =\sum_{\vec{n}\in\NN^d}
  \Proba\Bigl(
   \forall{}1\le{}i\le{}p\colon{}
   \GenPP{\ASet_i}\in\AtLeastSoManyPointsOn{\ASet_i,n_i}
  \Bigr),
\end{equation}
where we use the notation $\vec{n}=:(n_1,\dotsc,n_d)$.

\begin{lemma}
\label{lem_gstate_moment_bound}
For $p\in\NN$ and disjoint $\ASet_1,\dotsc,\ASet_p\in\BoundedBorelOf{\NonEmptyParticles}$,
\begin{equation}
\label{eq_gstate_moment_bound}
 \Expect\Bigl(\prod_{i=1}^p \Xi(\ASet_i)\Bigr)
 \le \lambda^p \prod_{i=1}^p \mu(\ASet_i).
\end{equation}
\end{lemma}

\begin{proof}
Applying~\eqref{eq_moment_sum_probas}, then stochastic domination from Lemma~\ref{lem_stoch_dom_palm}, applying~\eqref{eq_moment_sum_probas} again and writing out the moment of the Poisson process yields the bound.
\end{proof}

\begin{lemma}
\label{lem_spec_state_moment_diff}
Let $p\in\NN$ and let $\ASet_1,\dotsc,\ASet_p\in\BoundedBorelOf{\NonEmptyParticles}$
be disjoint.
Suppose that
$\BSet\in\BoundedBorelOf{\NonEmptyParticles}$ satisfies
$\BSet\supseteq{}\ASet:=\bigcup_{i=1}^p \ASet_i$ and
let $\chi\in\Configs{\BSet^c}$. Then
\begin{equation}
\label{eq_spec_state_moment_diff}
 \Bigl| \Expect\Bigl(\prod_{i=1}^p \Gibbs{}(\ASet_i)\Bigr)
-\Expect\Bigl(\prod_{i=1}^p \Gibbs{}(\ASet_i)
   \mathrel{\Big|} 
   \Gibbs{\BSet^c} = \chi\Bigr) \Bigr|
 \le\lambda^p\Bigl(\prod_{i=1}^p \mu(\ASet_i)\Bigr)
C_\Psi\exp(-C_2d(\Psi,\Gamma^c)),
\end{equation}
where $C_\Psi:=\sum_{i=1}^{q}C_1(\diam(\Psi_i))$.
\end{lemma}

\begin{proof}
Applying~\eqref{eq_moment_sum_probas} and then~\eqref{eq_increasing_spec_state} gives
\begin{align*}
 &\Bigl| \Expect\Bigl(\prod_{i=1}^p \Gibbs{}(\ASet_i)\Bigr)
-\Expect\Bigl(\prod_{i=1}^p \Gibbs{}(\ASet_i)
   \mathrel{\Big|} 
   \Gibbs{\BSet^c} = \chi\Bigr) \Bigr|
\\
&=  \sum_{\vec{n}\in\NN^d}
  \Bigl|
   \Proba\Bigl(
    \forall{}1\le{}i\le{}p\colon{}
    \Gibbs{\ASet_i}\in\AtLeastSoManyPointsOn{\ASet_i,n_i}\Bigr)
  -\Proba\Bigl(\forall{}1\le{}i\le{}p\colon{}
    \Gibbs{\ASet_i}\in\AtLeastSoManyPointsOn{\ASet_i,n_i}
    \mid{}\Gibbs{\BSet^c}=\chi\Bigr)
   \Bigr|
\\
&\le \sum_{\vec{n}\in\NN^d}
   \Proba\Bigl(
    \forall{}1\le{}i\le{}p\colon{}
    \Poisson{\lambda\mu}\in\AtLeastSoManyPointsOn{\ASet_i,n_i}
   \Bigr)
   \sum_{i=1}^{p}C_1(\diam(\ASet_i))\exp(-C_2 d(\ASet,\BSet^c)).
\end{align*}
Conclude by another application of~\eqref{eq_moment_sum_probas} and writing out the Poisson moment.
\end{proof}

\begin{lemma}
\label{lem_gstate_moment_diff_bound}
For $p,q\in\NN$ and disjoint $\ASet_1,\dotsc,\ASet_p,\CSet_1,\dotsc,\CSet_q\in\cB_b$,
\begin{multline}
\label{eq_gstate_moment_diff_bound}
 \Biggl| \Expect
   \Bigl(\prod_{i=1}^p \Gibbs{}(\ASet_i)\Bigr)
   \Bigl(\prod_{j=1}^q \Gibbs{}(\CSet_j)\Bigr)
  -\Expect\Bigl(\prod_{i=1}^p \Gibbs{}(\ASet_i)\Bigr)
  \Expect\Bigl(\prod_{i=1}^q \Gibbs{}(\CSet_i)\Bigr)
 \Biggr|
\\
 \le
 \lambda^{p+q}
 \Bigl(\prod_{i=1}^p \mu(\ASet_i)\Bigr)
 \Bigl(\prod_{i=1}^q \mu(\CSet_i)\Bigr)
 \min(C_{\ASet},C_{\CSet})
 \exp(-C_2 d(\ASet,\CSet))
 ,
\end{multline}
where $\ASet:=\bigcup_{i=1}^p \ASet_i$, $C_{\ASet}:=\sum_{i=1}^{p}C_1(\diam(\ASet_i))$, $\CSet:=\bigcup_{i=1}^q \CSet_i$ and\\
$C_{\CSet}:=\sum_{i=1}^{q}C_1(\diam(\CSet_i))$.
\end{lemma}

\begin{proof}
Let $\BSet$ be a large sphere such that
$d(\ASet\cup{}\CSet,\BSet^c) > d(\ASet,\BSet^c)$.
Hence,
\begin{equation}
\label{eq_gstate_moment_bound_enclosing}
 d(\ASet,\CSet\cup\BSet^c) = d(\ASet,\CSet).
\end{equation}

By the DLR-equation~\eqref{eq_dlr},
\begin{align*}
 \Biggl|&\Expect
   \Bigl(\prod_{i=1}^p \Gibbs{}(\ASet_i)\Bigr)
   \Bigl(\prod_{j=1}^q \Gibbs{}(\CSet_j)\Bigr)
 -\Expect\Bigl(\prod_{i=1}^p \Gibbs{}(\ASet_i)\Bigr)
  \Expect\Bigl(\prod_{i=1}^q \Gibbs{}(\CSet_i)\Bigr)
 \Biggr| \\
 &\le  \int\Biggl|\Expect\Bigl(\prod_{i=1}^p \Gibbs{}(\ASet_i)\Bigl|\Bigr. 
      \Gibbs{\CSet\cup\BSet^c}=\chi\Bigr)
    -\Expect\Bigl(\prod_{i=1}^p \Gibbs{}(\ASet_i)\Bigr)
   \Biggr|\Bigl(\prod_{j=1}^q \chi(\CSet_j)\Bigr)  \,\Proba(\Gibbs{\CSet\cup\BSet^c}\in\dd\chi).
\intertext{
Applying~\eqref{eq_spec_state_moment_diff} with $\BSet$ replaced by
$\BSet\setminus\CSet$ and noting that $(\BSet\setminus\CSet)^c=\CSet\cup\BSet^c$,
we bound
}
 &\le{}
 \lambda^p
 \Bigl(\prod_{i=1}^p \mu(\ASet_i)\Bigr)
  C_{\ASet}\exp(-C_2 d(\ASet,\CSet\cup\BSet^c))
 \int\Bigl(\prod_{j=1}^q \chi(\CSet_j)\Bigr)
 \,\Proba(\Gibbs{\CSet\cup\BSet^c}\in\dd\chi)
 \\
 &=
 \lambda^p
 \Bigl(\prod_{i=1}^p \mu(\ASet_i)\Bigr)
 C_{\ASet}\exp(-C_2 d(\ASet,\CSet\cup\BSet^c))
 \Expect\Bigl(\prod_{j=1}^q \Gibbs{}(\CSet_j)\Bigr)
 \\
 &\le
 \lambda^{p+q}
 \Bigl(\prod_{i=1}^p \mu(\ASet_i)\Bigr)
 \Bigl(\prod_{j=1}^q \mu(\CSet_j)\Bigr)
 C_{\ASet}\exp(-C_2 d(\ASet,\CSet)),
\end{align*}
where we use~\eqref{eq_gstate_moment_bound_enclosing} and~\eqref{eq_gstate_moment_bound} to obtain the final inequality.
The improved bound $\min(C_{\ASet},C_{\CSet})$ follows from the symmetry in $\ASet$ and $\CSet$.
\end{proof}

\begin{proof}[Proof of Theorem~\ref{thm_exp_decorrelation}]
Let $\ASet_{nj}$, $n,j\in\NN$, be a {\em dissection system} in
$\bN$~\cite[p.\ 20]{Kallenberg17}. Let $k\in\NN$ and let
$\alpha_k(\cdot):=\Expect \Xi^k(\cdot)$ denote
the {\em $k$\textsuperscript{th} moment measure} of $\Xi$.
Let $\alpha_k=\alpha'_k+\alpha''_k$ be the Lebesgue decomposition
(\cite[Corollary 1.29]{Kallenberg17}) of $\alpha_k$ with
respect to $\mu^k$, that is, $\alpha'_k$ is absolutely continuous
with respect to $\mu^k$ while $\alpha''_k$ and $\mu^k$ are mutually
singular.
Define
\begin{align*}
g_k(K_1,\dotsc,K_{k}):=\limsup_{n\to\infty}\sum_{j_1,\dotsc,j_{k}\in\NN}
\mu^{k}(\ASet_{n\vec{j}})^{-1}\alpha_{k}(\ASet_{n\vec{j}})
\I\{(K_1,\dotsc,K_{k})\in \ASet_{n\vec{j}}\},
\end{align*}
where we write $\vec{j}:=(j_1,\dotsc,j_{k})$ and
$\ASet_{n,\vec{j}}:=\ASet_{nj_1}\times\cdots\times\ASet_{nj_{k}}$
and where we set $a/0:=0$ for all $a\in\R$.
Outside the {\em generalised diagonal}
\begin{align*}
  D_{k}:=\{(K_1,\dotsc,K_{k}) \in (\NonEmptyParticles)^{k}\mid{}
\text{there exist $\{i,j\}\subseteq\{1,\dotsc,n\}$ with $K_i = K_j$}\}
\end{align*}
the function $g_k$ coincides with
\begin{align*}
g_k^{\ne}(K_1,\dotsc,K_{k}):=\limsup_{n\to\infty}\sideset{}{^{\ne}}\sum_{j_1,\dotsc,j_{k}\in\NN}
\mu^{k}(\ASet_{n\vec{j}})^{-1}\alpha_{k}(\ASet_{n\vec{j}})
\I\{(K_1,\dotsc,K_{k})\in \ASet_{n\vec{j}}\},
\end{align*}
where the superscript $\ne$ indicates summation over $k$-tuples with
distinct entries.
Since $\Xi$ is simple, the measure $\alpha^{(k)}$ is the restriction
of $\alpha_{k}$ to the complement of $D_{k}$, see~\cite[Exercise 6.9]{LastPenrose17} and the proof
of~\cite[Theorem 6.13]{LastPenrose17}. Moreover, Lemma~\ref{lem_gstate_moment_bound}
shows that $\alpha^{(k)}$ is absolutely continuous with respect to $\mu^k$.
Therefore we obtain from the proof of~\cite[Theorem 1.28]{Kallenberg17}
and~\eqref{eq_correlation_GNZ} that the above superior limits
are actually limits for $\mu^k$--a.e.\ $(K_1,\dotsc,K_{k})$
and, moreover, that
\begin{align}
\label{eq_correlation_density}
 \CorrFun{k}(K_1,\dotsc,K_{k}) = g_k(K_1,\dotsc,K_{k})
 ,\quad\alpha^{(k)}
 \text{--a.e.\ $(K_1,\dotsc,K_{k})\in(\NonEmptyParticles)^{k}$}
.
\end{align}

Let $p,q\in\NN$.
Using~\eqref{eq_correlation_density} for $k\in\{p+q,p,q\}$ and combining this
with Lemma~\ref{lem_gstate_moment_diff_bound},
we obtain for
$\alpha^{(p+q)}\text{--a.e.\ $(K_1,\dotsc,K_{p+q})\in(\NonEmptyParticles)^{p+q}$}$
that
\begin{align}\notag
 |&\CorrFun{p+q}(K_1,\dotsc ,K_{p+q})-
  \CorrFun{p}(K_1,\dotsc ,K_p)
  \CorrFun{q}(K_{p+1},\dotsc ,K_{p+q})|
\\ \notag
&\le\lambda^{p+q}
 \limsup_{n\to\infty}\sideset{}{^{\ne}}
  \sum_{j_1,\dotsc,j_{p+q}\in\NN}
  \min(p,q)C_1(1)
  \exp(-C_2 d(\cup^p_{i=1}\ASet_{nj_i},\cup^{p+q}_{i=p+1}\ASet_{nj_i}))
  \I\{(K_i)_{i=1}^{p+q}\in \ASet_{n\vec{j}}\}
\\\label{e2345}
&=
 \lambda^{p+q}\min(p,q)C_1(1)
 \exp(-C_2 d(\{K_1,\dotsc,K_p\},\{K_{p+1},\dotsc,K_{p+q}\}))
 .
\end{align}
Here the inequality can be obtained from
Lemma~\ref{lem_gstate_moment_diff_bound} as follows.
Fix $(K_1,\dotsc,K_{p+q})\in(\NonEmptyParticles)^{p+q}$ such that
$K_i\ne K_j$ for $i\ne j$. Then, for all sufficiently large $n\in\NN$,
there exists a unique $\vec{j}\in \NN^{p+q}$ with distinct entries such
that
$(K_1,\dotsc,K_{p+q})\in
\ASet_{n\vec{j}}=:\ASet_n(K_1,\dotsc,K_{p+q})$.
As $n\to\infty$ we have
$\ASet_n(K_1,\dotsc,K_{p+q})\downarrow\{K_1,\dotsc,K_{p+q}\}$.
Moreover the diameter of $\ASet_n(K_1,\dotsc,K_{p+q})$ (with respect
to the product of the Hausdorff metric) tends to $0$.
These facts do also imply the identity~\eqref{e2345}.
Indeed, we just need combine them with definition~\eqref{eq_pseudodistance} of
$d(\cdot,\cdot)$ and the triangle inequality.
\end{proof}

\section{Asymptotic properties of $U$-statistics}
\label{sec_applications}

In this section we fix an admissible Gibbs process $\Xi$ as in
Definition~\ref{def_admissible_gibbs}.

For $n\in\NN$, let $W_{n} := \left[-\frac{1}{2}n^{1/d},\frac{1}{2}n^{1/d}
\right]^{d}$
be the centred cube of volume $n$, $\cC^d_n := z^{-1}(W_n)$ and
$\xi_n:=\xi_{W_n}$, for $\xi\in\bN$.
Let $\Xi_n := \Xi_{W_n}$ and $\Xi_n^c:= \Xi_{W_n^c}$ be the
restriction of the Gibbs process to $\cC^d_n$ and $(\cC^d_n)^c$ respectively.
For a given mapping $F\colon \bN\to\R$, we are interested in the asymptotic
properties of $F(\Xi_n)$ as $n\to\infty$.
We focus on special mappings $F$ introduced next.

\subsection{Admissible $U$-statistics}

A function $h\colon (\NonEmptyParticles )^{k} \rightarrow \R$ is called {\em symmetric}
if $h(K_{1},\dotsc,K_{k}) = h(K_{\pi(1)},\dotsc,K_{\pi(k)})$,
for all $K_{1},\dotsc,K_{k} \in \NonEmptyParticles$
and every permutation $\pi$ of $k$ elements.
It is translation invariant,
if $h(K_{1},\dotsc,K_{k}) = h(\theta_{x}K_{1},\dotsc,\theta_{x}K_{k})$,
for all $K_1,\dotsc,K_k \in\NonEmptyParticles $ and $x \in \R^{d}$.
Given a measurable symmetric and translation invariant function $h$ we can define
\begin{equation}
\label{Fh_function}
 F_h(\xi):=
 \frac{1}{k!}\int h(K_1,\dotsc, K_k)\,\xi^{(k)}(\dd(K_{1},\dotsc,K_{k})),\quad \xi\in\bN.
\end{equation}
In fact, the functions $F_h(\Xi)$ and $F_h(\Xi_n)$
are {\em $U$-statistics} of {\em order} $k$, cf.~\cite{RS}
and~\cite[Chapter 12]{LastPenrose17}.
Define
\begin{equation}
\label{eq_u_statistic_def}
 T(K,\xi):=
 \frac{1}{k!}\int h(K,K_2,\dotsc, K_k)\,\xi^{(k-1)}(\dd(K_{2},\dotsc,K_{k})),
\quad (K,\xi) \in \cC^{(d)} \times \bN,
\end{equation}
where the case $k=1$ has to be read as $T(K):=h(K)$.
Then
\begin{equation*}
F_h(\xi):=\int T(K,\xi)\,\xi(\dd K),\quad \xi\in\bN.
\end{equation*}
The authors of~\cite{BYY} call $T$ a {\em score function}.

\begin{definition}
\label{def_admissible_ustatistic}
\rm
Let $k\in\NN$ and let
$h\colon (\NonEmptyParticles )^{k} \rightarrow \R$
be measurable, symmetric and translation invariant.
Then $F_h$ in~\eqref{Fh_function} is called an {\em admissible}
function (of order $k$) if $h(K_{1},\dotsc,K_{k})=0$, whenever either
\begin{equation}
\label{eq_func_interaction_bound}
 \max_{2 \leq i \leq k} d_{H}(K_{i},K_{1}) > r,
\end{equation}
for some given $r>0$, or when $K_{i} = K_{1}$, for some $i\in\{2,\dotsc,k\}$.
If, moreover,
\begin{equation}
\label{eq_h_sup_norm_bound}
 \|h \|_{\infty}:=
 \sup_{K_{1},\dotsc,K_{k} \in \NonEmptyParticles} |h(K_{1},\dotsc,K_{k})|
 < \infty,
\end{equation}
then
$F_h(\Xi_n)$, $n\in\NN$, is called an admissible $U$-statistic of order $k$ (of the Gibbs process $\Xi_n$).
\end{definition}

\begin{example}
\label{ex_facet2}
\rm Consider Example~\ref{ex_facet1}
and recall that
$\bN_{\tilde{V}}:=\{\xi\in\bN\mid{} \xi(\tilde{V}^c)=0\}$.
For $\xi\in \bN_{\tilde{V}}$ and $j\in\{1,\dots,d\}$, define using $Q_j$ from~\eqref{eq_potq}
\begin{equation}\label{facU}
 G_j(\xi):=
 \frac{1}{j!}\int_{\tilde{V}^j} Q_j(K_1,\dotsc,K_{j})
 \,\xi^{(j)}(\dd(K_1,\dotsc,K_{j})).
\end{equation}
Then $G_j$ is an admissible function with $r=2R$,
cf.~\eqref{eq_bounded_set}.
In the special case of $d=2$, facets are
segments in the plane and $G_2(\xi) $ is the total number of
intersections between segments in $\operatorname{supp}\xi$ having different
orientation. For $d=3$, facets are thin circular plates and $G_2(\xi )$
is the total length of intersections of pairs of facets in
$\operatorname{supp}\xi$.
\end{example}

Admissible functions satisfy a moment condition introduced in~\cite[Definition 1.8]{BYY}.

\begin{proposition}\label{prop_pmoment_cond}
If $F_h$ is an admissible function of order $k$ and $p\in\NN$, then
\begin{equation}
\label{eq_pmoment_cond}
 \sup_{n\in\NN}
 \sup_{1\leq q\leq p}
 \sup_{K_1,\dotsc ,K_q\in \cC_n^d}
  \Expect_{K_1,\dotsc ,K_q}\left[\max\{|T(K_1,\Xi_n)|, 1\}^p\right]
 < \infty,
\end{equation}
where $T$ is from~\eqref{eq_u_statistic_def} and
the inner supremum is an essential supremum with respect to
the $q$\textsuperscript{th} factorial moment measure of $\Xi$.
\end{proposition}

\begin{proof}
Let $q\in\{1,\dotsc,p\}$. Since $h$ is admissible, property~\eqref{eq_h_sup_norm_bound} implies that
\begin{equation*}
 \max \{|\ScoreFunc(K,\xi)|,1 \}^{p} \leq g_2(K,\xi),
 \quad (K,\xi)\in \NonEmptyParticles\times\bN,
\end{equation*}
where
\begin{equation*}
 g_2(K,\xi):= c\max\Bigl\{1, \left(\xi(B(K,r))^{k-1}\right)^{p}\Bigr\}
\end{equation*}
with $c:=\max\{1,\|h \|_{\infty}/k!\}$. In the following we argue
for $\mu^q$--a.e.\ $(K_1,\dotsc,K_q)\in (\NonEmptyParticles)^q$.
Since $\kappa\le 1$, Lemma~\ref{lem_stoch_dom_palm} shows that
$\Proba_{K_1,\dotsc,K_q}^{\,!}$ is stochastically dominated by
the distribution of $\Pi_{\lambda\mu}$. Therefore,
\begin{align*}
 \Expect_{K_1,\dotsc,K_q}\left[\max\{|T(K_1,\Xi_n)|, 1\}^p\right]
 &\le \Expect[g_2(K_1,\Pi_{\lambda\mu}+\delta_{K_1}+\dotsb+\delta_{K_q})]
 \\&\le c\, \Expect[(q+\Pi_{\lambda\mu}(B(K_1,r)))^{p(k-1)}].
\end{align*}
Using the inequality $(a+b)^{p(k-1)}\le 2^{p(k-1)-1}(a^{p(k-1)}+b^{p(k-1)})$,
for $a,b>0$, we obtain that
\begin{equation*}
 \Expect_{K_1,\dotsc,K_q}\left[\max\{|T(K_1,\Xi_n)|, 1\}^p\right]
 \le c2^{p(k-1)-1}\, (q^{p(k-1)-1}+\Expect[\Pi_{\lambda\mu}(B(K_1,r))^{p(k-1)}]).
\end{equation*}
The random variable $\Pi_{\lambda\mu}(B(K_1,r))$ has a Poisson distribution
with parameter
\begin{align*}
\BE[\Pi_{\lambda\mu}(B(K_1,r))]
&=\lambda\int \I\{d_H(K,K_1)\le r\}\,\mu(\dd K)\\
&=\lambda\iint \I\{d_H(K+x,K_1)\le r\}\,\BQ(\dd K)\, \dd x.
\end{align*}
By~\eqref{eq_centering} and~\eqref{eq_bounded_set} we may assume
that $K_1\subset B(z(K_1),R)$.
Assume that $K\in \NonEmptyParticles$ satisfies
$K\subset B(\mathbf{0},R)$ and that $\|x-z(K_1)\|>R$.
It follows from the definition of the Hausdorff distance
that $d_H(K+x,K_1)\ge \|x-z(K_1)\|-R$. Hence, uniformly in $K_1$ under our assumptions, we obtain the finite bound
\begin{align*}
 \BE[\Pi_{\lambda\mu}(B(K_1,r))]
 &\le \lambda\int \I\{\|x-z(K_1)\|\le R\}\, \dd x
 +\lambda\int \I\{R<\|x-z(K_1)\|\le r+R\}\, \dd x\\
 &= \lambda\int \I\{\|x\|\le r+R\}\, \dd x.
\end{align*}
Thus, the assertion follows from
the moment properties of a Poisson random variable.
\end{proof}

\subsection{Factorization of weighted mixed moments}
\label{sec_fact_wmm}

In this subsection we study an admissible pair $(\Xi ,T)$
defined as follows.

\begin{definition}\label{def_admissible_pair}\rm
We call $(\Xi ,T)$ an admissible pair if $\Xi$ is an admissible Gibbs
process with $\lambda<\PercThreshold{d}$ and the score function $T$
corresponds to an admissible function $F_h$,
cf.\ Definition~\ref{def_admissible_gibbs} and
Definition~\ref{def_admissible_ustatistic}.
\end{definition}

\begin{example}\rm
\label{exampl}
The Gibbs facet process from Example~\ref{ex_facet1} together with the
admissible function $G_j$ from Example~\ref{ex_facet2}
forms an admissible pair, for each $j\in\{1,\dots,d\}$.
\end{example}

Given $n,p,k_1, \dots,k_{p}\in\NN$ and $K_1,\dots ,K_{p}\in \NonEmptyParticles$, we define the {\em weighted mixed moment}
\begin{multline}\label{eq_mixed_moment}
m^{(k_1,\dotsc ,k_{p})}(K_1,\dotsc ,K_{p};n)
\\
:=\CorrFun{p}(K_1,\dotsc ,K_{p})
   \int_\bN
    \left(
     \ScoreFunc(K_1,\xi_{n})^{k_1}\dotsm\ScoreFunc(K_{p},\xi_{n})^{k_{p}}
    \right)
  \,\Proba_{K_1,\dotsc,K_{p}}(\dd\xi).
\end{multline}
In the following all equations and inequalities involving Palm distributions
and correlation functions are to be understood in the a.e.-sense
with respect to the appropriate factorial moment measures
of $\Xi$.

\begin{definition}\label{def_factorize_approximately}\rm
We adapt
the terminology of~\cite{BYY} and say that weighted mixed moments have fast decay of correlations,
if there exist constants $a_l,b_l>0$, $l\in\NN$,
such that
\begin{multline}\label{eq_moment_factorization}
\begin{aligned}
  &\Big|m^{(k_1,\dotsc ,k_{p+q})}(K_1,\dotsc,K_{p+q};n)
  \\
  &- m^{(k_1,\dotsc ,k_{p})}(K_1,\dotsc,K_{p};n)
     m^{(k_{p+1},\dotsc ,k_{p+q})}(K_{p+1},\dotsc,K_{p+q};n)
     \Big|
\end{aligned}
\\
\le
 b_t \exp(-a_t d(\{K_1,\dotsc,K_{p}\},\{K_{p+1},\dotsc ,K_{p+q}\})),
\end{multline}
for all $n,p,q,k_1, \dots,k_{p+q}\in\NN$ and for all $K_1,\dots ,K_{p+q}\in z^{-1}(W_n)$,
where $t:=\sum_{i=1}^{p+q}k_{i}$.
\end{definition}

We intend to use Theorem~\ref{thm_exp_decorrelation} to show
that~\eqref{eq_moment_factorization} holds in our context.  The method
from~\cite{BYY} is used and transformed step by step from point processes on
$\R^{d}$ to particle processes.  Recall that $\prec$ is the total
order of $\NonEmptyParticles$ introduced in Section~\ref{sec_dap} and
that intervals with respect to $\prec$ are in
$\BorelOf{\NonEmptyParticles}$.
In particular, for $K\in\NonEmptyParticles$,
$(-\infty,K)=\{L\in\NonEmptyParticles\mid{}\,L \prec K\}$.

Let $o$ be the zero-measure, i.e., $o(\Psi)=0$, for all
$\Psi \in \cB(\NonEmptyParticles )$.
Abbreviate $[l] := \{1,\dotsc,l\}$, for $l\in\NN$.
We define a difference operator for a measurable function
$\psi:\bN \rightarrow \R$, $l \in \NN \cup \{ 0\}$ and
$K_{1},\dotsc,K_{l} \in \NonEmptyParticles $ by
\begin{equation} \label{eq_diff_operator_def}
 D^{l}_{K_{1},\dotsc,K_{l}} \psi(\xi)
  :=
  \begin{cases}
   \sum_{J \subseteq [l]}(-1)^{l-|J|}
  \psi(\xi_{(-\infty,K_{\ast})} +
    \sum_{j \in J}\delta_{K_{j}})
    &\text{if $l>0$,}
   \\
   \psi(o)
    &\text{if $l=0$,}
  \end{cases}
\end{equation}
where $K_{\ast} := \min \{K_{1},\dotsc,K_{l} \}$ with respect to $\prec$.
We say that $\psi$ is $\prec$-continuous at $\infty$, if
$\lim_{K \uparrow \R^{d}}\psi(\xi_{(-\infty,K)}) = \psi(\xi)$, for all $\xi \in \bN$.

We use the following \emph{factorial moment expansion} (FME) proved
in~\cite[Theorem 3.1]{BMS} on a general Polish space.
For stronger results in the special case of a Poisson process
we refer to~\cite{Last14} and~\cite[Chapter 19]{LastPenrose17}.

\begin{theorem}
Let $\psi\colon\bN\rightarrow \R$ be $\prec$-continuous at $\infty$.
Assume that, for all $l \in \NN$,
\begin{align}
\label{eq_FMEcondition1}
 \int_{(\NonEmptyParticles )^{l}}
  \Expect^{!}_{K_{1},\dotsc,K_{l}}
  [|D^{l}_{K_{1},\dotsc,K_{l}}\psi(\Xi)| ]
  \CorrFun{l}(K_{1},\dotsc,K_{l})
  \,\mu^l( \dd(K_{1}, \dotsc, K_{l} ))
 < \infty
\intertext{and}
\label{eq_FMEcondition2}
 \lim_{l \rightarrow \infty}
  \frac{1}{l!}
   \int_{(\NonEmptyParticles )^{l}}
    \Expect^{!}_{K_{1},\dotsc,K_{l}}
     [D^{l}_{K_{1},\dotsc,K_{l}}\psi(\Xi) ]
    \CorrFun{l}(K_{1},\dotsc,K_{l})
    \,\mu^l( \dd(K_{1}, \dotsc, K_{l}) )
   = 0.
\end{align}
Then, $\Expect [\psi(\Xi)]$ has the FME
\begin{equation}
\label{eq_FME}
 \Expect [\psi(\Xi)]
  =
  \psi(o) +
  \sum_{l=1}^{\infty}
  \frac{1}{l!}
  \int_{(\NonEmptyParticles )^{l}}
    D^{l}_{K_{1},\dotsc,K_{l}}\psi(o)
    \CorrFun{l}(K_{1},\dotsc,K_{l})
  \,\mu^l( \dd(K_{1}, \dotsc, K_{l}) )
.
\end{equation}
\end{theorem}

\noindent For an admissible pair $(\ScoreFunc, \Xi )$,
$K_{1},\dotsc,K_{p} \in \NonEmptyParticles $ and $ \xi \in \bN $, set
\begin{align}
\label{eq_mixed_product1}
  \psi_{k_{1},\dotsc,k_{p}}(K_{1},\dotsc,K_{p};\xi)
   :=
   \prod_{i=1}^{p}\ScoreFunc(K_{i},\xi)^{k_{i}},
\\
\label{eq_mixed_product2}
  \psi^{!}_{k_{1},\dotsc,k_{p}}(K_{1},\dotsc,K_{p};\xi)
   :=
    \prod_{i=1}^{p}
     \ScoreFunc\Bigl(K_{i},\xi+\sum_{j=1}^{p}\delta_{K_{j}}\Bigr)^{k_{i}}
,
\end{align}
with $k_{1},\dotsc,k_{p} \geq 1$.
It holds that
$
\Expect_{K_{1},\dotsc,K_{p}}[\psi(\Xi_n)]=
\Expect^{!}_{K_{1},\dotsc,K_{p}}[\psi^{!}(\Xi_n)]$.
Given $p\in\NN$ and $K_1,\ldots,K_p\in \NonEmptyParticles$
we denote by $\CorrFun{l}^{(K_{1},\dotsc,K_{p})}$
the $l$\textsuperscript{th} correlation function of $\BP^{!}_{K_{1},\dotsc,K_{p}}$.
Further we let $(\BP^{!}_{K_{1},\dotsc,K_{p}})^{!}_{L_{1},\dotsc,L_{l}}$,
$L_1,\ldots,L_l\in \NonEmptyParticles$
denote the reduced Palm distributions  of $\BP^{!}_{K_{1},\dotsc,K_{p}}$.
It is easy to show that
\begin{align}\label{reccorr}
\CorrFun{p}(K_{1},\dotsc,K_{p})
 \CorrFun{l}^{(K_{1},\dotsc,K_{p})}(L_{1},\dotsc,L_{l})
  =
 \CorrFun{p+l}(K_{1},\dotsc,K_{p},L_{1},\dotsc,L_{l})
 \,,
\end{align}
and
\begin{align}\label{recPalm}
(\BP^{!}_{K_{1},\dotsc,K_{p}})^{!}_{L_{1},\dotsc,L_{l}}=
\BP^{!}_{K_{1},\dotsc,K_{p},L_1,\ldots,L_l}.
\end{align}

\begin{lemma}\label{lem_FME}
For distinct $K_{1},\dotsc,K_{p} \in \cC^{(d)}$, $k_{1},\dotsc,k_{p},n\in\NN$,
$t_p:=\sum_{i=1}^pk_i$ and $k$ the order of the $U$-statistic,
the functional $\psi^{!}$ admits the FME
\begin{multline*}
\Expect^{!}_{K_{1},\dotsc,K_{p}}
 [\psi^{!}_{k_{1},\dotsc,k_{p}}(K_{1},\dotsc,K_{p};\Xi_n)]
 =
 \psi^{!}_{k_{1},\dotsc,k_{p}}(K_{1},\dotsc,K_{p};o)
 \\
 +
 \sum_{l=1}^{t_p(k-1)}
  \frac{1}{l!}
   \int_{(\NonEmptyParticles )^{l}}
    D^{l}_{L_{1},\dotsc,L_{l}}
    \psi^{!}_{k_{1},\dotsc,k_{p}}(K_{1},\dotsc,K_{p};o)
    \CorrFun{l}^{(K_{1},\dotsc,K_{p})}(L_{1},\dotsc,L_{l})
    \,\mu^l(\dd(L_{1},\dotsc,L_{l})).
\end{multline*}
\end{lemma}

\begin{proof}
We abbreviate $\psi_{k_{1},\dotsc,k_{p}}(K_{1},\dotsc,K_{p};\Xi)$ by
$\psi(K_{1},\dotsc,K_{p};\Xi)$.  The radius bound $r$
from~\eqref{eq_func_interaction_bound} for the function $h$ implies
that $\psi^{!}$ is $\prec$-continuous at $\infty$.  In~\cite[Lemma
5.1]{BYY} it is shown that $\psi^{!}$ is the sum of $U$-statistics
of orders not larger than $t_p(k-1)$.
Thus, for $l \in (t_p(k-1),\infty)$
and all $L_{1},\dotsc,L_{l} \in \NonEmptyParticles$, we have
\begin{equation}
\label{eq_high_order_diff_stat}
 D^{l}_{L_{1},\dotsc,L_{l}}
  \psi^{!}(K_{1},\dotsc,K_{p};\xi)
 = 0
.
\end{equation}
This implies that~\eqref{eq_FMEcondition1}, for $l \in (t_p(k-1),\infty)$, and~\eqref{eq_FMEcondition2} are satisfied for $\psi^{!}$ from~\eqref{eq_mixed_product2}.
We need to verify~\eqref{eq_FMEcondition1}, for $l \in [1,t_p(k-1)]$.
For $L_{1},\dotsc,L_{l} \in \NonEmptyParticles $, $\xi\in\bN$ and $J \subseteq [l]$, set
\begin{equation*}
 \xi_{J} := \xi_{(-\infty,L_{\ast})} + \sum_{j\in J} \delta_{L_{j}},
\end{equation*}
where $L_{\ast}:=\min \{L_{1},\dotsc,L_{l} \}$ and $(-\infty,L_{\ast})$
are with respect to the order $\prec$.

The difference operator $D^{l}_{L_{1},\dotsc,L_{l}}$ vanishes like
in~\eqref{eq_high_order_diff_stat} as soon as
$L_{m} \notin \bigcup_{i=1}^{p}B(K_{i},2r)$, for some $m \in [l]$.  To
prove this, expand~\eqref{eq_diff_operator_def} to obtain
\begin{multline*}
 D^{l}_{L_{1},\dotsc,L_{l}}
 \psi^{!}(K_{1},\dotsc,K_{p};\xi)
 =
  \sum_{J \subseteq [l], m \notin J}
   (-1)^{l-|J|}
    \psi^{!}(K_{1},\dotsc,K_{p};\xi_{J})
 \\
  +
  \sum_{J \subseteq [l], m \notin J}
   (-1)^{l-|J|-1}
    \psi^{!}(K_{1},\dotsc,K_{p};\xi_{J\cup \{m \}})
 =
  0
 \,,
\end{multline*}
since, for fixed $J\subseteq [l]$ and $m \notin J $,
$\psi^{!}(K_{1},\dotsc,K_{p};\xi_{J})=\psi^{!}(K_{1},\dotsc,K_{p};\xi_{J\cup\{m \}})$.
Then, we have
\begin{align*}
\psi^{!}(K_{1},\dotsc,K_{p};\xi_{J})
&\leq
  \prod_{i=1}^{p}
   \Vert h \Vert^{k_{i}}_{\infty}
    \left(
     \xi \left( \bigcup_{i=1}^{p}B(K_{i},2r)\right) + |J| + p
    \right)^{k_{i}(k-1)}
\\
&\leq
  \Vert h \Vert^{t_p}_{\infty}
   \left(
    \xi \left( \bigcup_{i=1}^{p}B(K_{i},2r)\right) + |J| + p
   \right)^{t_p(k-1)}
.
\end{align*}
Using this for difference operator we get
\begin{align}
\nonumber
| D^{l}_{L_{1},\dotsc,L_{l}}
 \psi^{!}(K_{1},\dotsc,K_{p};\xi) |
&\leq
  \Vert h \Vert^{t_p}_{\infty}
   \sum_{J \subseteq [l]}
    \left(
     \xi\left(\bigcup_{i=1}^{p}B(K_{i},2r)\right) + |J| + p
    \right)^{t_p(k-1)}
\\
\label{eq_bound_difference}
&\leq
  \Vert h \Vert^{t_p}_{\infty}
  2^{l}
   \left(
    \xi\left(\bigcup_{i=1}^{p}B(K_{i},2r)\right) + l + p
   \right)^{t_p(k-1)}
.
\end{align}
Using~\eqref{recPalm}, the defining equation~\eqref{eq_palm}
and~\eqref{eq_bound_difference} results in
\begin{align*}
&\frac{1}{l!}
  \int_{(\NonEmptyParticles )^{l}}
   (\Expect^{!}_{K_{1},\dotsc,K_{p}})^{!}_{L_{1},\dotsc,L_{l}}
    [|D^{l}_{L_{1},\dotsc,L_{l}}
     \psi^{!}(K_{1},\dotsc,K_{p};\Xi_{n}) |]
\\
&\times
  \CorrFun{l}^{(K_{1},\dotsc,K_{p})}(L_{1},\dotsc,L_{l})
  \,\mu^l(\dd(L_{1},\dotsc,L_{l}))
\\
&=
\frac{1}{l!}
 \int_{(\NonEmptyParticles )^{l}}
  \Expect^{!}_{K_{1},\dotsc,K_{p},L_{1},\dotsc,L_{l}}
   [|D^{l}_{L_{1},\dotsc,L_{l}}\psi^{!}(K_{1},\dotsc,K_{p};\Xi_{n}) |]
\\
&\times
 \CorrFun{l}^{(K_{1},\dotsc,K_{p})}(L_{1},\dotsc,L_{l})
 \,\mu^l(\dd(L_{1},\dotsc,L_{l}))
\\
\leq
&\Vert h \Vert^{t_p}_{\infty} 2^{l}
 \Expect_{K_{1},\dotsc,K_{p}}
  \left[
   \Xi
    \left(
     \bigcup_{i=1}^{p}B(K_{i},r)
    \right)^{l}
    \left(
     \Xi
      \left(
       \bigcup_{i=1}^{p}B(K_{i},r)
      \right)+l+p
    \right)^{t_p(k-1)}
  \right]
\,.
\end{align*}
Since $\Xi$ has all moments under the Palm distribution the finiteness of the last term and hence the validity of the condition  for $l \in [1,t_p(k-1)]$ follows. This justifies the FME expansion.
\end{proof}

\begin{theorem}
\label{thm_score_factorizes}
Let $(\ScoreFunc,\Xi )$ be an admissible pair. Then the weighted moments have fast decay of correlations.
\end{theorem}

\begin{proof}
Let $p,q,k_1,\dots ,k_{p+q}\in\NN$ be fixed. Let $u:=\max(4R+R_\varphi,r)$, with $R_\varphi$ being the finite interaction range of the admissible particle process as outlined in Definition~\ref{def_admissible_gibbs} and taking into account the particle size from~\eqref{eq_bounded_set}, as well as~\eqref{eq_func_interaction_bound}.
Given $n\in\NN$, $K_{1},\dotsc,K_{p+q}\in z^{-1}(W_{n})$, we set
\begin{equation*}
s
 :=
 d(\{K_{1},\dotsc,K_{p} \},
   \{K_{p+1},\dotsc, K_{p+q} \})
 .
\end{equation*}
Without loss of generality we assume that $s \in (8u,\infty)$.  Put
$t$ as in Definition~\ref{def_factorize_approximately} and $t_p$ as in
Lemma~\ref{lem_FME} respectively and let
$t_q:=\sum_{i=p+1}^{p+q}k_{i}$.

Then, using Lemma~\ref{lem_FME},~\eqref{eq_high_order_diff_stat}
and~\eqref{reccorr}, we obtain
\begin{align*}
&
 m^{(k_{1},\dotsc,k_{p+q})}(K_{1},\dotsc,K_{p+q};n)
=
 \Expect^{!}_{K_{1},\dotsc,K_{p+q}}
  [\psi^{!}(K_{1},\dotsc,K_{p+q};\Xi_{n})]
   \CorrFun{p+q}(K_{1},\dotsc,K_{p+q})
\\
&
 =
  \sum_{l=0}^{t(k-1)} \frac{1}{l!}
   \int_{(\cC_{n}^{d})^{l}}
    D^{l}_{L_{1},\dotsc,L_{l}}
     \psi^{!}(o)
      \CorrFun{l+p+q}(K_{1},\dotsc,K_{p+q},L_{1},\dotsc,L_{l})
   \,\mu^l(\dd(L_{1},\dotsc, L_{l})) \\
&
 =
  \sum_{l=0}^{t(k-1)} \frac{1}{l!}
   \int_{(\bigcup_{i=1}^{p+q}B(K_{i},2u))^{l}}
    D^{l}_{L_{1},\dotsc,L_{l}}\psi^{!}(o)
     \CorrFun{l+p+q}(K_{1},\dotsc,K_{p+q},L_{1},\dotsc,L_{l})
   \,\mu^l(\dd(L_{1},\dotsc, L_{l}))
\,.
\end{align*}
Let $\Psi_{j,l}:=\Bigl(\bigcup_{i=1}^{p}B(K_{i},2u)\Bigr)^{j} \times\Bigl(\bigcup_{i=1}^{q}B(K_{p+i},2u)\Bigr)^{l}$.
Then, using the FME from Lemma~\ref{lem_FME},
\begin{align*}
m^{(k_{1},\dotsc,k_{p+q})}
 &
  (K_{1},\dotsc,K_{p+q};n)
\\
=
 &
  \sum_{l=0}^{t(k-1)}
   \frac{1}{l!}
    \sum_{j=0}^{l}
     \frac{l!}{j!(l-j)!}
   \int_{\Psi_{j,l-j}}
    D^{l}_{L_{1},\dotsc,L_{l}}
     \psi^{!}(K_{1},\dotsc,K_{p+q};o)
\\
 &
  \times
   \CorrFun{l+p+q}(K_{1},\dotsc,K_{p+q},L_{1},\dotsc,L_{l})
  \,\mu^l(\dd (L_{1},\dotsc, L_{l}))
\\
=
 &
  \sum_{l=0}^{t(k-1)}
   \sum_{j=0}^{l}
    \frac{1}{j!(l-j)!}
    \int_{\Psi_{j,l-j}}
     \sum_{J \subseteq [l]}
      (-1)^{l-|J|}
      \psi^{!}\Bigl(K_{1},\dotsc,K_{p+q}; \sum_{j \in J} \delta_{L_{j}}\Bigr)
\\
 &
  \times
   \CorrFun{l+p+q}(K_{1},\dotsc,K_{p+q},L_{1},\dotsc,L_{l})
  \,\mu^l(\dd(L_{1},\dotsc, L_{l}))
\,.
\end{align*}
To compare the $(p+q)$\textsuperscript{th} mixed moment with the product of the $p$-th
and $q$\textsuperscript{th} mixed moments we use a factorization that holds for
$L_{1},\dotsc,L_{j} \in \bigcup_{i=1}^{p}B(K_{i},2u)$ and
$L_{j+1},\dotsc,L_{l} \in \bigcup_{i=1}^{q}B(K_{p+i},2u)$.
If
\begin{equation*}
  K\in \bigcup_{i=1}^{p}B(K_{i},2u),\quad
  L\in \bigcup_{i=1}^{q}B(K_{p+i},2u),
\end{equation*}
then $K\cap L = \emptyset$.
Hence,
\begin{equation}
\label{eq_factorization_functional}
  \psi^{!}(K_{1},\dotsc,K_{p+q};\sum_{i=1}^{l}\delta_{L_{i}})
   =
  \psi^{!}\Bigl(K_{1},\dotsc,K_{p};\sum_{i=1}^{j}\delta_{L_{i}}\Bigr)
  \psi^{!}\Bigl(K_{p+1},\dotsc,K_{p+q};\sum_{i=j+1}^{l}\delta_{L_{i}}\Bigr)
\,.
\end{equation}
Using~\eqref{eq_factorization_functional} and similar steps as in the case of the $(p+q)$\textsuperscript{th} mixed moment we work with the product of $p$\textsuperscript{th} and $q$\textsuperscript{th} mixed moments.
\begin{align*}
m^{(k_{1},\dotsc,k_{p})}
 &
  (K_{1},\dotsc,K_{p};n)
m^{(k_{p+1},\dotsc,k_{p+q})}(K_{p+1},\dotsc,K_{p+q};n)
\\
={}&
  \Expect^{!}_{K_{1},\dotsc,K_{p}}
   [\psi^{!}(K_{1},\dotsc,K_{p},\Xi_{n})]
  \Expect^{!}_{K_{p+1},\dotsc,K_{p+q}}
   [\psi^{!}(K_{p+1},\dotsc,K_{p+q},\Xi_{n})]
 \\
 &\times
  \CorrFun{p}(K_{1},\dotsc,K_{p})
  \CorrFun{q}(K_{p+1},\dotsc,K_{p+q})
\\
 =&
  \sum_{l_{1},l_{2}=0}^{\infty}\int_{\Psi_{l_{1},l_{2}}}
    D^{l_{1}}_{L_{1},\dotsc,L_{l_{1}}}
     \psi^{!}(K_{1},\dotsc,K_{p};o)
    D^{l_{2}}_{N_{1},\dotsc,N_{l_{2}}}
     \psi^{!}(K_{p+1},\dotsc,K_{p+q};o)
 \\
 &\times
  \CorrFun{l_{1}+p}(K_{1},\dotsc,K_{p},L_{1},\dotsc,L_{l_{1}})
  \CorrFun{l_{2}+q}(K_{p+1},\dotsc,K_{p+q},N_{1},\dotsc,N_{l_{2}})
 \\
 &\times
  \mu^{l_{1}}(\dd(L_{1},\dotsc,L_{l_{1}}))
  \,\mu^{l_{2}}(\dd(N_{1},\dotsc,N_{l_{2}}))
 \\
 =&
  \sum_{l_{1},l_{2}=0}^{\infty}
   \frac{1}{l_{1}!l_{2}!}
    \int_{\Psi_{l_{1},l_{2}}}
     \sum_{J_{1}\subseteq[l_{1}], J_{2}\subseteq[l_{2}]}
      (-1)^{l_{1}+l_{2}-|J_{1}|-|J_{2}|}
 \\
 &\times
  \psi^{!}\Bigl(K_{1},\dotsc,K_{p};\sum_{i \in J_{1}}\delta_{L_{i}}\Bigr)
  \psi^{!}\Bigl(K_{p+1},\dotsc,K_{p+q};\sum_{i \in J_{2}}\delta_{N_{i}}\Bigr)
 \\
 &\times
  \CorrFun{l_{1}+p}(K_{1},\dotsc,K_{p},L_{1},\dotsc,L_{l_{1}})
  \CorrFun{l_{2}+q}(K_{p+1},\dotsc,K_{p+q},N_{1},\dotsc,N_{l_{2}})
 \\
 &\times
  \mu^{l_{1}}(\dd(L_{1},\dotsc,L_{l_{1}}))
  \,\mu^{l_{2}}(\dd(N_{1},\dotsc,N_{l_{2}}))
 \\
 =&
 \sum_{l=0}^{t(k-1)} \sum_{j=0}^{l}
  \frac{1}{j!(l-j)!}\int_{\Psi_{j,l-j}}
   \sum_{J_{1} \subseteq [j], J_{2} \subseteq [l] \setminus [j]}
    (-1)^{l-|J_{1}|-|J_{2}|}
 \\
 &\times
  \psi^{!}(K_{1},\dotsc,K_{p};
   \sum_{i \in J_{1}}\delta_{L_{i}})
    \psi^{!}\Bigl(K_{p+1},\dotsc,K_{p+q};\sum_{i \in J_{2}}\delta_{L_{i}}\Bigr)
 \\
 &\times
  \CorrFun{j+p}(K_{1},\dotsc,K_{p},L_{1},\dotsc,L_{j})
  \CorrFun{l-j+q}(K_{p+1},\dotsc,K_{p+q},L_{j+1},\dotsc,L_{l})
 \\
 &\times
  \mu^l(\dd(L_{1},\dotsc,L_{l}))
 \\
 =&
  \sum_{l=0}^{t(k-1)}\sum_{j=0}^{l}
   \frac{1}{j!(l-j)!}\int_{\Psi_{j,l-j}}
    \sum_{J \subseteq [l]}(-1)^{l-|J|}
 \\
 &\times
  \psi^{!}\Bigl(K_{1},\dotsc,K_{p+q};\sum_{i \in J}\delta_{L_{i}}\Bigr)
    \CorrFun{j+p}(K_{1},\dotsc,K_{p},L_{1},\dotsc,L_{j})
 \\
 &\times
  \CorrFun{l-j+q}(K_{p+1},\dotsc,K_{p+q},L_{j+1},\dotsc,L_{l})
  \,\mu^l(\dd(L_{1},\dotsc,L_{l}))
 \,.
\end{align*}

Altogether we have, using $c_1,c_2$ from~\eqref{eq_exp_decorrelation}, that
\begin{align*}
 \Bigl|m^{(k_{1},\dotsc,k_{p+q})}&(K_{1},\dotsc,K_{p+q};n)-
   m^{(k_{1},\dotsc,k_{p})}(K_{1},\dotsc,K_{p};n)
   m^{(k_{p+1},\dotsc,k_{p+q})}(K_{p+1},\dotsc,K_{p+q};n)\Bigr|
\\
\leq&
 \sum_{l=0}^{t(k-1)}\sum_{j=0}^{l}\sum_{J \subseteq [l]}
  \frac{(-1)^{l-|J|}}{j!(l-j)!}
   \int_{\Psi_{j,l-j}}
    \psi^{!}\Bigl(K_{1},\dotsc,K_{p+q};\sum_{i \in J}\delta_{L_{i}}\Bigr)
\\
&\times
 \Bigl|\CorrFun{j+p}(K_{1},\dotsc,K_{p},L_{1},\dotsc,L_{j})
  \CorrFun{l-j+q}(K_{p+1},\dotsc,K_{p+q},L_{j+1},\dotsc,L_{l})\Bigr.
\\
&-
 \Bigl.\CorrFun{l+p+q}(K_{1},\dotsc,K_{p+q},L_{1},\dotsc,L_{l})\Bigr|
 \,\mu^l(\dd(L_{1},\dotsc,L_{l}))
\\
\leq&
 \sum_{l=0}^{t(k-1)}
  \sum_{j=0}^{l}
   \sum_{J \subseteq [l]}
    \frac{(-1)^{l-|J|}}{j!(l-j)!}
     \int_{\Psi_{j,l-j}}
      \psi^{!}\Bigl(K_{1},\dotsc,K_{p+q};\sum_{i \in J}\delta_{L_{i}}\Bigr)
\\
&\times
 \lambda^{l+p+q}\min(j+p,l-j+q)c_1
\\
&\times
  \exp(-c_2d(\{K_{1},\dotsc,K_{p},L_{1},\dotsc,L_{j}\},
       \{K_{p+1},\dotsc,K_{p+q},L_{j+1},\dotsc,L_{l}\})
\\
&\times
 \mu^l(\dd(L_{1},\dotsc,L_{l}))
\,.
\end{align*}
Using
\begin{align*}
 \ScoreFunc(K,\xi)\I\{\xi(\cC^{(d)})=n\}
  \leq
   \frac{n^{k-1}}{k} \Vert h \Vert_{\infty}
\,,
\end{align*}
we have
\begin{equation*}
\sum_{J \subseteq [l]}
 \Bigl|\psi^{!}\Bigl(K_{1},\dotsc,K_{p+q};\sum_{i \in J}\delta_{L_{i}}\Bigr)\Bigr|
 \leq
 2^{l}
 \left(
  \Vert h \Vert_{\infty}\frac{|p+q+l|^{k-1}}{k}
 \right)^{t}
.
\end{equation*}
The difference of weighted mixed moments is finally bounded by
\begin{align*}
&\Bigl|m^{(k_{1},\dotsc,k_{p+q})}(K_{1},\dotsc,K_{p+q};n)-
  m^{(k_{1},\dotsc,k_{p})}(K_{1},\dotsc,K_{p};n)
  m^{(k_{p+1},\dotsc,k_{p+q})}(K_{p+1},\dotsc,K_{p+q};n)\Bigr|
\\
&\leq
 \sum_{l=0}^{t(k-1)}\sum_{j=0}^{l}
  \left(
   \Vert h \Vert_{\infty}
   \frac{|p+q+l|^{k-1}}{k}
  \right)^{t}
  \frac%
   {\exp(-8u)2^{l}(-1)^{l-|J|}}
   {j!(l-j)!}
  \mu(B(\{\mathbf{0}\},3u))^{l}
\\
 &\times
  \lambda^{l+p+q}\min(j+p,l-j+q)c_1
  \exp(-c_2d( \{K_{1},\dotsc,K_{p}\},\{K_{p+1},\dotsc,K_{p+q}\})
\,.
\end{align*}
As $\min(j+p,l-j+q)\le{}l+p+q\le{}kt$, we obtain the desired constants for exponential decay depending only on $t$ and the attributes of the admissible pair.
\end{proof}

\subsection{Limit theorems}

In this subsection we prove mean and variance asymptotics of
admissible $U$-statistics $F_n:=F_h(\Xi_n)$, $n\in\NN$, as well as a
central limit theorem. The approaches to the
  proofs come from~\cite{BYY} and they can be generalized from point
  processes in $\R^d$ to particle processes.  In
  Theorem~\ref{thm_expect_limit} this generalization is possible
  thanks to our version of the $p$-moment condition in
  Proposition~\ref{prop_pmoment_cond}. The method of the proof of
  Theorem~\ref{thm_CLT} is standard and its step by step
  generalization to particle processes is omitted.

To proceed we need good versions of the correlation functions
$\rho_1$ and $\rho_2$ and the first and second order Palm distributions.
The measure $\BE[\int \I\{(K,\Xi)\in\cdot\}\,\Xi(\dd K)]$
is invariant under (joint) translations.
By~\cite[Theorem 7.6]{Kallenberg17} there exists an
translation invariant version of the correlation function $\rho_1$.
Moreover, the first order Palm distributions can be chosen
in an invariant way, that is
\begin{align}\label{inv1}
\BP_{K-x}(\theta_x\Xi\in\cdot)=\BP_{K}(\Xi\in\cdot),\quad
(K,x)\in \NonEmptyParticles\times\R^d.
\end{align}
By the same argument we can assume that $\rho_2$ is translation invariant
and
\begin{align}\label{inv2}
\BP_{K-x,L-x}(\theta_x\Xi\in\cdot)=\BP_{K,L}(\Xi\in\cdot),\quad
(K,L,x)\in \NonEmptyParticles\times\NonEmptyParticles\times\R^d.
\end{align}

For the weighted mixed
moments in~\eqref{eq_mixed_moment} we denote some special cases as
follows.  For $K,L\in\cC^{(d)}$ we set
\begin{align*}
 m_{(1)}(K) &:= \BE_K[T(K,\Xi)]\rho_1(K),\\
 m_{(2)}(K,L)&:= \BE_{K,L}[T(K,\Xi)T(L,\Xi)]\rho_2(K,L),\\
 m_{(1,2)}(K) &:=\BE_K[T^2(K,\Xi)]\rho_1(K).
\end{align*}
Further for $n\in\NN,\,x\in\R^d$ we abbreviate
\begin{align*}
m_{(1)}(K;n,x)& :=\BE_K[T(K,\Xi_n^x)]\rho_1(K),\\
m_{(2)}(K,L;n,x)& :=\BE_{K,L}[T(K,\Xi_n^x)T(L,\Xi_n^x)]\rho_2(K,L),
\end{align*}
where $\Xi_n^x:=\Xi\cap (W_n-n^{\frac{1}{d}}x)$.

\begin{theorem}\label{thm_expect_limit}
Let $(T,\Xi )$ be an admissible pair. Then it holds that
\begin{equation}
\label{eq_expect_limit}
 \lim_{n\rightarrow\infty}\frac{1}{n} \BE F_n
 = \int_{\cC^d_0} m_{(1)}(K)\,\BQ(\dd K)
 ,
\end{equation}
\begin{multline}
\label{eq_var_limit}
 \lim_{n\rightarrow\infty}\frac{1}{n}\V{F_n}
 =
 \int_{\cC^d_0}m_{(1,2)}(K)\,\BQ(\dd K)
\\
 + \int_{(\cC^d_0)^2}\int_{\R^d}
    (m_{(2)}(K,L+x)-m_{(1)}(K)m_{(1)}(L))
   \,\dd x\,\BQ(\dd K)\,\BQ(\dd L)
  <\infty
.
\end{multline}
\end{theorem}

\begin{proof}
From~\eqref{eq_palm},~\eqref{eq_correlation_GNZ} and~\eqref{eq_measure_mu}, we deduce that
\begin{align*}
 \BE F_n
 &=\int_{\cC_n^d}\BE_K T(K,\Xi_n)\rho_1(K)\,\mu(\dd K)
 \\
 &=\int_{\cC_{\mathbf{0}}^{(d)}}\int_{W_n}
  \BE_{K+x}[T(K+x,\Xi_n)]\rho_1(K+x)\,\dd x\,\BQ(\dd K).
\end{align*}
By equation~\eqref{inv1} and translation invariance of $T$,
$$
\BE_{K+x}[T(K+x,\Xi)]=\BE_{K}[T(K,\Xi)].
$$
Together with the translation invariance of $\rho_1$ this gives
\begin{equation*}
 \int_{\cC_{\mathbf{0}}^{(d)}}\int_{W_n}\BE_{K+x}[T(K+x,\Xi)]\rho_1(K+x)\,\dd x\,\BQ(\dd K)
 =
 n\int_{\cC^{(d)}}m_{(1)}(K)\,\BQ(\dd K)
.
\end{equation*}
To prove~\eqref{eq_expect_limit} it remains to show that
\begin{equation*}
  \frac{1}{n}
   \int_{\cC_{\mathbf{0}}^{(d)}}
    \int_{W_n}
     |
       \BE_{K+x}[T(K+x,\Xi_n)]
      -\BE_{K+x}[T(K+x,\Xi)]
     |
    \,\dd x
   \,\rho_1(K)\,\BQ(\dd K)
\end{equation*}
tends to zero as $n\rightarrow\infty$. The function $\rho_1$ is bounded.
For $K\in\cC^d_0$, $K\subseteq B(\mathbf{0},R)$ fixed and $x\in W_n$, we use~\eqref{eq_func_interaction_bound} to obtain $T(K+x,\Xi_n)=T(K+x,\Xi)$ whenever $d(x,\partial W_n)\geq 2R$ for
the distance from $x$ to the boundary of $W_n$ holds.
The 1-moment condition~\eqref{eq_pmoment_cond} implies the existence of some $0<a<\infty$ such that $\BE_{K+x}[|T(K+x,\Xi_n)|]\leq a$ and $\BE_{K+x}[|T(K+x,\Xi)|]\leq a$, uniformly in $n\in\NN$ and $K+x\in\cC_n^d$.
Since
\begin{equation*}
 \lim_{n\rightarrow\infty}
 \frac{1}{n}\mathcal{L}_d
 \{x\in W_n\mid{}\,d(x,\partial W_n)\leq 2R\}
 = 0,
\end{equation*}
the first assertion of the theorem is proven.

By a standard point process calculation we obtain for the second moment
\begin{multline*}
 \BE F_n^2
=
 J_1+J_2
:=
 \int_{\cC_{\mathbf{0}}^{(d)}}
  \int_{W_n}
   \BE_{K+u}
    T^2(K+u,\Xi_n)
    \rho_1(K+u)
  \,\dd u
 \,\BQ(\dd K)
\\
+
 \int_{(\cC_{\mathbf{0}}^{(d)})^2}
  \int_{(W_n)^2}
   \BE_{K+u,L+v}[T(K+u,\Xi_n)T(L+v,\Xi_n)]
   \rho_2(K+u,L+v)
  \,\dd u\,\dd v
 \,\BQ(\dd K)\,\BQ(\dd L)
.
\end{multline*}
Then
\begin{equation*}
 \lim_{n\rightarrow\infty}\frac{J_1}{n}
 =
 \int_{\cC^{(d)}}m_{(1,2)}(K)\,\BQ(\dd K)
\end{equation*}
is obtained analogously to the mean value asymptotics above using the 2-moment condition~\eqref{eq_pmoment_cond}.

In the second term $J_2$ we use the substitutions $x=n^{-\frac{1}{d}}u$ and $z=v-u$, obtaining
\begin{align*}
\frac{J_2}{n}
&=
 \int_{(\cC_{\mathbf{0}}^{(d)})^2}
  \int_{W_1}
   \int_{W_n-n^{\frac{1}{d}}x}
    \BE_{K+n^{\frac{1}{d}}x,L+z+n^{\frac{1}{d}}x}
    [T(K+n^{\frac{1}{d}}x,\Xi_n)
     T(L+z+n^{\frac{1}{d}}x,\Xi_n)]
\\
&\times
  \rho_2(K+n^{\frac{1}{d}}x,L+z+n^{\frac{1}{d}}x)
 \,\dd z\,\dd x
\,\BQ(\dd K)\,\BQ(\dd L)
\\
&=
 \int_{(\cC_{\mathbf{0}}^{(d)})^2}
  \int_{W_1}
   \int_{W_n-n^{\frac{1}{d}}x}
    \BE_{K,L+z}
     [T(K,\Xi^x_n)
      T(L+z,\Xi^x_n)]
     \rho_2(K,L+z)
    \,\dd z\,\dd x
  \,\BQ(\dd K)\,\BQ(\dd L),
\end{align*}
where we have used the invariance of $\rho_2$,
\eqref{inv2}, the invariance of $T$ and
$(\Xi+n^{\frac{1}{d}}x)_n-n^{\frac{1}{d}}x=\Xi^x_n.$
Since $\V{F_n}=\BE F_n^2-(\BE F_n)^2$, we investigate the expression $\frac{1}{n}(J_2-(\BE F_n)^2)$. It takes the form
\begin{equation}
\label{eq_var_proc}
 \int_{(\cC^d_0)^2}
  \int_{W_1}
   \int_{W_n-n^{\frac{1}{d}}x}
    (m_{(2)}(K,L+z;n,x)-
     m_{(1)}(K;n,x)
     m_{(1)}(L+z;n,x))
   \,\dd z\,\dd x
  \,\BQ(\dd{}K)
 \,\BQ(\dd{}L)
.
\end{equation}
Splitting the innermost integral in~\eqref{eq_var_proc} into the two terms
\begin{equation}
\label{eq_var_split}
 \int_{W_n-n^{\frac{1}{d}}x}
  (\dots )
 \,\dd z
=
 \int_{W_n-n^{\frac{1}{d}}x}
  (\dots ){\bf 1}
  \{|z|\leq M\}
  \,\dd z +
 \int_{W_n-n^{\frac{1}{d}}x}
  (\dots ){\bf 1}
  \{|z|> M\}
 \,\dd z,
\end{equation}
for an arbitrary $M>0$, we observe that the part of~\eqref{eq_var_proc} corresponding to the first term of~\eqref{eq_var_split}, i.e.
\begin{gather*}
 \int_{(\cC^d_0)^2}
  \int_{W_1}
   \int_{W_n-n^{\frac{1}{d}}x}
    (m_{(2)}(K,L+z;n,x)-
     m_{(1)}(K;n,x)
     m_{(1)}(L+z;n,x))
 \\
  \times
  {\bf 1}\{|z|\leq M\}
  \,\dd z\,\dd x
  \,\BQ(\dd{}K)
  \,\BQ(\dd{}L)
 ,
\end{gather*}
converges to
\begin{equation*}
 \int_{(\cC_{0}^{d})^2}
 \int_{\R^d}
  (m_{(2)}(K,L+x)-m_{(1)}(K)m_{(1)}(L))
 \,\dd x
 \,\BQ(\dd K)\,\BQ(\dd L)
 ,
\end{equation*}
when first $n\rightarrow\infty$ and then $M\rightarrow\infty$. Using~\eqref{eq_moment_factorization},
the absolute value of the second term in~\eqref{eq_var_split} can be bounded uniformly in $n$ by
$$b_2\int_{|z|>M}\exp(-a_2d_H(K,L+z))\dd z,$$
which tends to zero when $M\rightarrow\infty$.
Thus the part of~\eqref{eq_var_proc} corresponding to the second term of~\eqref{eq_var_split}, i.e.
\begin{gather*}
 \int_{(\cC^d_0)^2}
  \int_{W_1}
   \int_{W_n-n^{\frac{1}{d}}x}
    (m_{(2)}(K,L+z;n,x)-
     m_{(1)}(K;n,x)
     m_{(1)}(L+z;n,x))
 \\
  \times
  {\bf 1}\{|z|> M\}
  \,\dd z\,\dd x
  \,\BQ(\dd{}K)
  \,\BQ(\dd{}L)
 ,
\end{gather*}
converges to zero. We can justify these limits analogously to Lemma 4.1 in~\cite{BYY}. The boundedness in~\eqref{eq_var_limit}
follows from the 2-moment condition~\eqref{eq_pmoment_cond} for the
first term and from~\eqref{eq_moment_factorization} for the second
term.
\end{proof}

\begin{theorem}\label{thm_CLT}
Let $(\ScoreFunc,\Xi )$ be an admissible pair.
If, for some $\beta\in (0,\infty)$,
\begin{equation}
\label{eq_variance_lower_bound}
 \liminf_{n\rightarrow\infty}
  \frac%
   {\V\,F_n}
   {n^\beta}
  >0,
 \end{equation}
then we have the CLT
\begin{equation}
\label{eq_CLT}
 \frac%
  {{F_n}-\Expect {F_n}}
  { (\V{F_n})^{1/2}}
 \xrightarrow[n\rightarrow\infty]{d}
  N(0,1)
 \,.
\end{equation}
\end{theorem}
\begin{proof}
  Denote $\bar{F}_n:=F_n-\Expect F_n$. The idea is to prove that the
  $l$\textsuperscript{th} order cumulants of
  $(\V{F_n})^{-1/2}\bar{F}_n$ vanish as $n\rightarrow\infty $ and $l$
  large. This follows by showing that~\eqref{eq_pmoment_cond}
  and~\eqref{eq_moment_factorization} imply volume order growth (i.e.,
  of order $O(n)$) for the $l$\textsuperscript{th} order cumulant of
  $\bar{F}_n$, $l\geq 2$, and using the
  assumption~\eqref{eq_variance_lower_bound}. Then~\eqref{eq_CLT}
  holds. The details are analogous to~\cite[p. 881-886]{BYY}, with the
  difference that deal with measures defined on $\cC_n^d$.
\end{proof}

\begin{remark}\label{r4.12}\rm Checking assumption
\eqref{eq_variance_lower_bound} is a problem in its own right.
In view of~\cite[Theorem 1.1]{XY15} it can be expected
that~\eqref{eq_variance_lower_bound} holds, whenever the activity
$\lambda$ is smaller than some critical branching intensity.
In particular this should apply to the facet processes from
Example~\ref{ex_facet1} and Example~\ref{ex_facet2}.
Remark~\ref{r3.5} shows that this critical branching intensity
can be smaller than the critical intensity of our
dominating Boolean model.
\end{remark}

\begin{remark}\label{r4.13}\rm It might be conjectured
that~\eqref{eq_variance_lower_bound} holds for any admissible pair $(\Xi,T)$
(see Definition~\ref{def_admissible_pair}), provided that $T$ is non-degenerate
in a suitable sense. A proof should be based on the specific properties
of disagreement percolation. We leave this as an interesting open problem,
which is beyond the scope of our paper.
\end{remark}

\section{Concluding remarks}
\label{sec_discussion}

The results of this paper have the potential for
several extensions. The disagreement coupling and its consequences,
for instance, can probably be derived for other
potentials (without a deterministic range)
and other spaces. A similar comment applies to
our results for admissible $U$-statistics. Moreover,
it can be expected that these results
can be extended to stabilizing functionals, as studied in~\cite{BYY}.
Again one would then obtain an improvement in the range of
possible activities; cf.~\eqref{eq_bound_new} and~\eqref{eq_bound_sy13}.

\bigskip

\noindent
{\bf Acknowledgement:}
This research started when then Ph.D. student J. Ve\v{c}e\v{r}a visited G. Last in late 2016, the support of KIT Karlsruhe and of the DFG-Forschergruppe FOR 1548 ``Geometry and Physics of Spatial Random Systems'' for this stay is highly appreciated. Further the authors were supported as follows: V. Bene\v{s} by Czech Science Foundation, project 19-04412S, J. Ve\v{c}e\v{r}a by Charles University, project SVV260454.
C. Hofer-Temmel thanks the Charles University, Faculty of Mathematics and Physics, and the KIT Karlsruhe for their hospitality.


\bigskip

\begin{thebibliography}{30}

\bibitem{BYY}
Blaszczyszyn, B., Yogeshwaran D., Yukich. J.E. (2019).
Limit theory for geometric statistics of point processes having fast decay of correlations.
{\em Ann. Probab.}
{\bf 47 (2)}, 835--895.

\bibitem{BMS}
Blaszczyszyn, B., Merzbach, E., Schmidt, V. (1997). A note on expansion for
functionals of spatial marked point processes.
{\em Statist. Probab. Lett.}
{\bf 36 (3)}, 299--306.

\bibitem{CSKM}
Chiu, S.N., Stoyan, D., Kendall, W.S., Mecke, J.\ (2013).
{\em Stochastic Geometry and its Applications.}
Third Edition, Wiley, Chichester.

\bibitem{CorMoeWaa}
Coeurjolly, J.-F., M{\o}ller, J., Waagepetersen, R.\ (2017).
A tutorial on Palm distributions for spatial point processes.
{\em Int. Stat. Rev. }
{\bf 83 (5)}, 404--420.

\bibitem{DVJ}
Daley, D.J., Vere-Jones, D.\ (2003/2008).
{\em An Introduction to the Theory of Point Processes.
Volume I: Elementary Theory and Methods,
Volume II: General Theory and Structure.} 2nd edn.\ Springer, New York.

\bibitem{DDG} Dereudre, D., Drouilhet, R., Georgii,
H.-O. (2012). Existence of Gibbsian point processes with geometry-dependent interactions.
{\em Probab. Theory Relat. Fields}
{\bf 153}, 643--670.

\bibitem{Dudley_RAP}
Dudley, R.M. (1989).
{\em Real Analysis and Probability.
The Wadsworth \& Brooks/Cole Mathematics Series.} Wadsworth \& Brooks/Cole Advanced Books \& Software.

\bibitem{Ferrari}
Ferrari, P.A., Fern\'{a}ndez, R., Garcia, N.L. (2002).
Perfect simulation for interacting point processes, loss networks and Ising models.
{\em Stoch. Process. Applic.}  {\bf 102}, 63--88.

\bibitem{FlimmelBenes} Flimmel, D., Bene\v{s}, V. (2018).
Gaussian approximation for functionals of Gibbs particle processes.
{\em Kybernetika} {\bf 54 (4)}, 765--777.

\bibitem{GeorKun}
Georgii, H.-O., K\"{u}neth, T. (1997). Stochastic order of point processes.
{\em J. Appl. Probab.}
{\bf 34 (4)}, 868--881.

\bibitem{GeorYoo05}
Georgii, H.-O., Yoo, H.J. (2005).
Conditional intensity and Gibbsianness of determinantal point processes
{\em J. Stat. Phys.} {\bf 118}, 55-–84.

\bibitem{Goue}
Gouere, J.B. \ (2008).
Subcritical regimes in the Poisson Boolean model of continuum percolation.
{\em Ann. Probab. } {\bf 36 (4)}, 1209--1220.

\bibitem{HTH}
Hofer-Temmel C., Houdebert P. (2019). Disagreement percolation for Gibbs ball models.
{\em Stochastic Process. Appl.} {\bf 129 (10)}, 3922--3940.

\bibitem{Jansen19}
Jansen, S. (2019).
Cluster expansions for Gibbs point processes.
{\em Adv. Appl. Probab.} {\bf 51 (4)}, 1129--1178.

\bibitem{Kallenberg}
Kallenberg, O.\ (2002).
{\em Foundations of Modern Probability.}
Second Edition, Springer, New York.

\bibitem{Kallenberg17}
Kallenberg, O.\ (2017).
\newblock {\it Random Measures, Theory and Applications.}
\newblock Springer, Cham.

\bibitem{Last14}
Last, G.\ (2014).
Perturbation analysis of Poisson processes.
{\em Bernoulli} {\bf 20}, 486--513.

\bibitem{LastPenrose17}
Last, G., Penrose, M.\ (2017).
{\em Lectures on the Poisson Process.}
Cambridge University Press, Cambridge.

\bibitem{Mase00}
Mase, S. (2000).
Marked Gibbs processes and asymptotic
normality of maximum pseudo‐likelihood estimators.
{\em Math.\ Nachr.} {\bf 209}, 151--169.

\bibitem{MaWaMe79}
Matthes, K.,  Warmuth, W., Mecke, J.\ (1979).
Bemerkungen zu einer Arbeit von Nguyen Xuan Xanh und Hans Zessin.
{\em Math.\ Nachr.} {\bf 88}, 117--127.

\bibitem{MeesterRoy96}
Meester, R., Roy, R.\ (1996).
{\em Continuum Percolation.} Cambridge University Press, Cambridge.

\bibitem{Penrose96}
Penrose, M.\ (1996).
Continuum percolation and Euclidean minimal spanning trees in high dimensions
{\em Ann. Appl. Probab.} {\bf 6}, 528--544.

\bibitem{RS}
Reitzner, M., Schulte, M.\ (2013).
Central limit theorems for $U$-statistics of Poisson point processes,
{\em  Ann. Probab.} {\bf 41,} 3879--3909.

\bibitem{Ruelle70}
Ruelle, D.\ (1970).
Superstable interactions in classical statistical mechanics.
{\em Comm.\ Math.\ Phys.} {\bf 18}, 127--159.

\bibitem{SW08}
Schneider R., Weil W.\ (2008).
{\em Stochastic and Integral Geometry.}
Springer, Berlin.

\bibitem{SY13}
Schreiber T., Yukich J.E.\ (2013).
Limit theorems for geometric functionals of Gibbs point processes.
{\em Ann. Inst. Henri Poincar\'{e} Probab. Stat.} {\bf 49}, 1158--1182.

\bibitem{Tor17}
Torrisi G.L.\ (2017).
Probability approximation of point processes with Papangelou conditional intensity.
{\em Bernoulli} {\bf 23}, 2210--2256.

\bibitem{VB17} Ve\v{c}e\v{r}a J., Bene\v{s} V.\ (2017).
Approaches to asymptotics for U-statistics of Gibbs facet processes.
{\em Statist. Probab. Lett.} {\bf 122}, 51--57.

\bibitem{VB16} Ve\v{c}e\v{r}a, J., Bene\v{s}, V.\ (2016).
Interaction processes for unions of facets, the asymptotic behaviour with
increasing intensity.
{\em Methodol. Comput. Appl. Probab.} {\bf 18 (4)},  1217--1239.

\bibitem{XY15}
Xia, A., Yukich, J.E.\ (2015).
Normal approximation for statistics of Gibbsian input in geometric probability.
{\em Adv. in Appl. Probab.}{\bf 25}, 934--972.

\bibitem{Ziesche18}
Ziesche, S.\ (2018).
Sharpness of the phase transition and lower bounds for the
critical intensity in continuum percolation on $\mathbb{R}^d$.
{\em Ann. Inst. Henri Poincar\'{e} Probab. Stat.} {\bf 54}, 866--878.

\end{thebibliography}
\end{document}